\documentclass[12pt]{article} 
\usepackage[letterpaper,margin=1in]{geometry}
\usepackage{fancyhdr,color}
\usepackage{tikz,graphicx,multicol}
\usepackage{amssymb,euscript,nicefrac,enumitem}
\usepackage{amsfonts,amsmath,amsthm} 
\usepackage{scrextend}
\usepackage{ytableau} 
    \ytableausetup{boxsize=1.0em,aligntableaux=center}
\usepackage[noend]{algpseudocode}
\usepackage{extarrows}
\usepackage{shuffle}
\usepackage{pdfpages}
\usepackage{algorithm}
\usepackage{arydshln}
\usetikzlibrary{cd}

\newtheorem{theorem}{Theorem}
\newtheorem{corollary}{Corollary}
\newtheorem{lemma}{Lemma}
\newtheorem{proposition}{Proposition}

\theoremstyle{definition}
\newtheorem{definition}{Definition}
\newtheorem{remark}{Remark}
\newtheorem{example}{Example}

\definecolor{myred}{rgb}{.7,.1,.1}
\definecolor{myblue}{rgb}{.1,.1,.7}
\definecolor{mygreen}{rgb}{.1,.6,.1}
\definecolor{mygray}{rgb}{0.5,0.5,0.5}
\definecolor{mymauve}{rgb}{0.58,0,0.82}
\definecolor{lablue1}{rgb}{0,0.9,0.93}
\definecolor{lablue2}{rgb}{0,0.5,1}

\def\gdot{\color{black!25}{\scriptstyle\cdot}}
\usepackage{listings}

\fancypagestyle{plain}{
  
  \fancyhf{}
  \fancyhead[L]{}
  \fancyhead[C]{}
  \fancyhead[R]{}
  \fancyfoot[C]{\footnotesize \thepage} 
}

\def\sym{\mathrm{Sym}}
\def\qsym{\mathrm{QSym}}
\def\nsym{\mathrm{NSym}}
\def\ncsym{\mathrm{NCSym}}
\def\ncqsym{\mathrm{NCQSym}}
\def\fqsym{\mathrm{FQSym}}

\def\setcomp{\mathsf{setcomp}}
\def\sort{\mathsf{sort}}
\def\hgt{\mathsf{ht}}
\def\len{\mathsf{len}}
\def\row{\mathsf{row}}
\def\col{\mathsf{col}}

\def\diag{\mathsf{diag}}
\def\push{\mathsf{push}}
\def\qshuf{\widetilde{\shuffle}}

\def\med{\text{Med}}

\def\rsrt{\mathsf{sort}_r}
\def\pus{\mathsf{push}}

\def\sd{\mathcal{SD}}
\def\dd{\mathcal{DD}}

\def\sld{\mathcal{SLD}}
\def\ldd{\mathcal{LDD}}
\def\lsr{\mathcal{LSR}}

\def\sdr{\mathsf{SDR}}

\DeclareMathAlphabet{\mathpzc}{OT1}{pzc}{m}{it}

\usepackage[backend=bibtex,maxbibnames=99]{biblatex}
\title{Powersum Bases in Quasisymmetric Functions and Quasisymmetric Functions in Non-commuting Variables}
\author{Anthony Lazzeroni}
\date{}

\makeatletter
\let\theauthors\@author
\makeatother

\begin{document}

\maketitle 
\thispagestyle{plain}
\abstract{We introduce a new $P$ basis for the Hopf algebra of quasisymmetric functions that refine the symmetric powersum basis. Unlike the quasisymmetric power sums of types 1 and 2, our basis is defined combinatorially: its expansion in  quasisymmetric monomial functions is given by fillings of matrices. This basis has a shuffle product, a deconcatenate coproduct, and has a change of basis rule to the quasisymmetric fundamental basis by using tuples of ribbons. We lift our quasisymmetric powersum $P$ basis to the Hopf algebra of quasisymmetric functions in non-commuting variables by introducing fillings with disjoint sets. This new basis has a shifted shuffle product and a standard deconcatenate coproduct, and certain basis elements agree with the fundamental basis of the Malvenuto-Reutenauer Hopf algebra of permutations. Finally we discuss how to generalize these bases and their properties by using total orders on indices.}

\section{Introduction}

The Hopf algebra of symmetric functions, denoted $\sym$ and indexed by integer partitions $\lambda$, is a very well known space for its connections in representation theory and other areas of mathematics. Some of the well studied bases are the \textit{monomial} basis $m_\lambda$, \textit{powersum} basis $p_\lambda$, and \textit{Schur} basis $s_\lambda$. In \cite{macdonald1998symmetric} the change of basis from $p_\lambda$ to $m_\mu$ is illustrated by fillings. In this work, a \textit{filling} $\mathsf{F}$ is a rectangular matrix with one non-zero entry in every row and no zero columns. The column (respectively, row) reading is a composition recording the sum of all entries in each column (respectively, row) denoted as $\col(\mathsf{F})$ (respectively, $\row(\mathsf{F})$). Thus the change of basis is combinatorially defined as 
\begin{equation} \label{eqn:ptom}
    p_\lambda = \sum_{\mathsf{F}\in \mathcal{A}(\lambda)}m_{\col(\mathsf{F})}
\end{equation}
where $\mathcal{A}(\lambda)$ is the set of all distinct fillings with row reading $\lambda$ and the column sum a partition. For example, $p_{(2,2,1)}=2m_{(2,2,1)}+2m_{(3,2)}+m_{(4,1)}+m_{(5)}$ since the fillings of $\mathcal{A}(2,2,1)$ are
\[ \footnotesize
\begin{array}{c|ccc}
2 & 2 & \gdot  & \gdot  \\
2 & \gdot  & 2 & \gdot  \\
1 & \gdot  & \gdot  & 1 \\
\hline
& 2 & 2 & 1
\end{array}
\qquad
\begin{array}{c|ccc}
2 & \gdot  & 2 & \gdot  \\
2 & 2 & \gdot  & \gdot  \\
1 & \gdot  & \gdot  & 1 \\
\hline
& 2 & 2 & 1
\end{array}
\qquad
\begin{array}{c|ccc}
2 & 2      & \gdot  & \gdot  \\
2 & \gdot  & 2      & \gdot  \\
1 & 1      & \gdot  & \gdot  \\
\hline
& 3 & 2 & 
\end{array}
\]\[\footnotesize
\begin{array}{c|ccc}
2 & \gdot  & 2      & \gdot  \\
2 & 2      & \gdot  & \gdot  \\
1 & 1      & \gdot  & \gdot  \\
\hline
& 3 & 2 & 
\end{array}
\qquad
\begin{array}{c|ccc}
2 & 2     & \gdot  & \gdot  \\
2 & 2     & \gdot  & \gdot  \\
1 & \gdot & 1      & \gdot  \\
\hline
& 4 & 1 & 
\end{array}
\qquad
\begin{array}{c|ccc}
2 & 2  & \gdot  & \gdot  \\
2 & 2  & \gdot  & \gdot  \\
1 & 1  & \gdot  & \gdot  \\
\hline
& 5 &  & 
\end{array}. 
\]
One of the important properties of the powersum basis is the Murnaghan-Nakayama rule which illustrates the product rule of $s_\lambda$ and $p_\mu$ expanded in terms of Schur functions. Giving, as corollary, the change of basis formula from the powersum basis to the Schur basis, which is important for its connection to the character table of $\mathfrak{S}_n$.

A space that contains $\sym$ is the Hopf algebra of quasisymmetric functions, $\qsym$, and is indexed by compositions $\alpha$. This space was defined in \cite{IraM.Gessel1984MultipartiteFunctions} by using $P$-partitions, which gives rise to the fundamental quasisymmetric functions $F_\alpha$. In \cite{Ballantine2020OnSums} the authors studied two quasisymmetric powersum bases (i.e. a basis of $\qsym$ that refines $p_\lambda$) $\Phi$ and $\Psi$ whose duals are the $\mathbf{\Phi}$ and $\mathbf{\Psi}$ bases in $\nsym$ which is defined in \cite{Gelfand1995NoncommutativeFunctions}. These bases of NSym aren't defined combinatorially, but via formal series. The $\Psi$ basis is most notable for the change of basis to the fundamental basis by using $P$-partitions as defined in \cite{Alexandersson2021P-PartitionsP-Positivity}. In \cite{aliniaeifard2021peak} the authors introduced the Shuffle basis $S_\alpha$, which also refines $p_\lambda$, is notable because $S_\alpha$ is an eigenvector under the theta map $\Theta$.

In this paper we define a quasisymmetric powersum basis $P_\alpha$ combinatorially by using fillings in an analog of \eqref{eqn:ptom}. Alternatively, $P_\alpha$ can be defined using a subposet of the refinement poset $\mathcal{P}$ on compositions.  This basis has a shuffle product, a deconcatenate coproduct, they refine the symmetric powersum basis and is dual (up to scaling) to the Zassenhaus functions $Z_\alpha$ in $\nsym$, the algebra of non-commutative symmetric functions, as defined in \cite{krob1997noncommutative}. All of these results are stated in Sections \ref{sec:basdef} and \ref{sec:scapwr}. We note that this powersum basis was independently defined in \cite{aliniaeifard2021p} using weighted generating functions of P-partitions. In Section \ref{sec:CoB}, we show a Murnaghan-Nakayama like change of basis rule from the quasisymmetric powersum basis to the quasisymmetric fundamental basis by use of tuples of ribbons. 

In Section \ref{sec:faminv}, we note that both the filling and the subposet notions have generalizations that depend on a total order $\preceq$ on the parts of the compositions. Hence we can define a whole family of quasisymmetric powersum bases, denoted as $P^{\succeq}_\alpha$ for different choices of $\succeq$, so that all the properties described are held for any choice of $\succeq$.

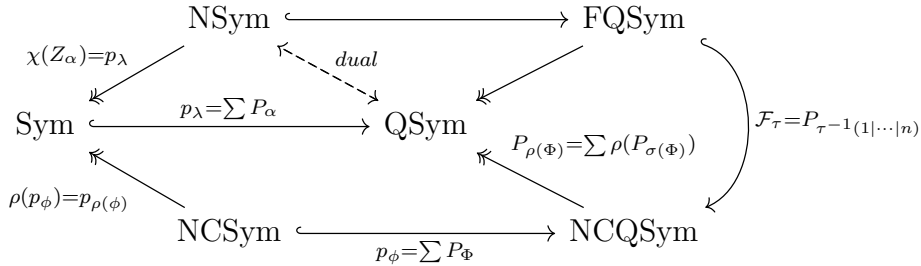
\begin{figure}[h] 
\begin{center}
\begin{tikzcd} 
 & \nsym \arrow[rr, hook] \arrow[dl, twoheadrightarrow, "\chi(Z_\alpha)=p_\lambda"'] & & \fqsym  \arrow[dl, twoheadrightarrow] \arrow[dd, hook, bend left=70, "\mathcal{F}_\tau=P_{\tau^{-1}(1|\cdots|n)}" ] \\
\sym \arrow[rr, hook, "p_\lambda = \sum P_\alpha"] & & \qsym \arrow[ul, dashleftarrow] \arrow[ul, dashrightarrow, "dual"'] \arrow[dr, twoheadleftarrow, "P_{\rho(\Phi)} = \sum \rho(P_{\sigma(\Phi)})" near start ]\\
 & \ncsym \arrow[ul, twoheadrightarrow, "{\rho(p_\phi)=p_{\rho(\phi)}}"] \arrow[rr, hook, "p_\phi = \sum P_\Phi"'] & & \ncqsym
\end{tikzcd}
\end{center}
\caption{Diagram of some Hopf Algebras related to $\qsym$}
\label{fig:alg}
\end{figure}

Finally, in Section \ref{sec:othralg} we show how one might use this method to define a powersum basis in other algebras.

Figure \ref{fig:alg} above summarizes the relationships between our new powersum bases for $\qsym$ and $\ncqsym$ and existing powersum bases in other algebras.

\subsection*{Acknowledgments}
The author would like to thank his advisor Amy Pang for all the help and guidance with the development of this project and the development of the author as a mathematician. The author would like to thank Aaron Lauve for this project, a lot of the code used, editing and guidance. The author would also like to thank Travis Scrimshaw for all of his helpful comments. Finally the author would like to thank the Sage community.

\section{Preliminaries for QSym}
Given an integer composition of $n$ of length $k$, $\alpha = (a_1,a_2,\ldots,a_k)$, let $m_i(\alpha)$ denote the multiplicities of the part size $i$ within $\alpha$, and denote the weak composition of size multiplicities by $m(\alpha)=(m_1(\alpha),\dotsc,m_n(\alpha))$. Let $\mathcal{P}$ be a poset on compositions with the cover relation $\alpha\lessdot\beta$ if $\beta$ is obtained from $\alpha$ by summing two adjacent parts, i.e. $\beta = (a_1,\dotsc,a_i+a_{i+1},\dotsc,a_k)$.

For the purpose of several proofs in Section \ref{sec:proj}, the standard bijection between compositions of $n$ and subsets of $[n-1]$ will be useful:
\[
(a_1,\ldots,a_k) \longleftrightarrow \{ a_1, a_1+a_2,\ldots, a_1+\cdots+a_{k-1} \}.
\]
Using this identification, we see $\alpha\lessdot\beta$ if and only if $\alpha = \beta \cup \{k\}$ for some $k\in[n-1]$.

Let $X=\{ x_1,x_2\ldots \}$ be a set of variables and $\alpha = (a_1,\ldots,a_k)$ be a composition of $n$. Then a generating function is quasisymmetric if, for every $k$ and $i_1<i_2<\cdots<i_k$, the coefficient of $x_{i_1}^{\alpha_1}x_{i_2}^{\alpha_2}\cdots x_{i_k}^{\alpha_k}$ is equal to the coefficient of $x_{1}^{\alpha_1}x_{2}^{\alpha_2}\cdots x_{k}^{\alpha_k}$. When $|X| = \infty$, the set of all quasisymmetric functions is a graded Hopf algebra denoted as $\qsym$. One of the more natural bases of $\qsym$ is the \textit{quasisymmetric monomial basis} $M_\alpha$, defined as
\[
M_\alpha = \sum_{i_1<i_2<\cdots<i_k}x_{i_1}^{\alpha_1}x_{i_2}^{\alpha_2}\cdots x_{i_k}^{\alpha_k}. 
\]
$\qsym$ is most notable for the \textit{quasisymmetric fundamental basis} $F_\alpha$ which proves useful for countless enumeration problems and also corresponds to the characters of the 0-Hecke algebra. It is defined as 
\[
F_\alpha = \sum_{\beta\leq\alpha} M_\beta.
\]

\section{Powersum functions in QSym} \label{sec:basdef}
\subsection{Fillings Interpretation}

A \textit{filling} of a composition $\alpha=(a_1,\dotsc, a_k)$, denoted $\mathsf{F}$, is a $k\times k$ matrix where each $a_i$ can be placed anywhere in row $i$. The \textit{column sum} of a filling, denoted $\col (\mathsf{F})$, is the composition $\beta = (b_1, \dotsc, b_l)$ such that the sum of all the entries in column $j$ is $b_j$. The \textit{row reading}, $\row(\mathsf{F})$, has a similar definition, i.e. $\row(\mathsf{F})=\alpha$. The following two fillings are important to the sequel.

\begin{remark}
Fillings can be defined more generally to generate other bases in $\sym$, see \cite{stanley1999enumerative}, but for the subject of this paper the definition above suffices.
\end{remark}

\begin{definition} \label{def:DD}
A \textit{Diagonal Descending} filling is a filling such that:
\begin{enumerate}
    \item entry $a_1$ is in the upper leftmost corner of the matrix.
    \item entry $a_i$ is in row $i$.
    \item \begin{enumerate}
        \item if $a_{i-1}\geq a_{i}$ then $a_{i}$ is directly below $a_{i-1}$ or in the southeast position of $a_{i-1}$
        \item if $a_{i-1}<a_i$ then $a_i$ is in the southeast position of $a_{i-1}$.
        \end{enumerate}
    \item row $i$ and $j$ can permute only if $a_i = a_j$ and $a_i$ is not in the same column as $a_j$.
\end{enumerate}
Denote the set of all diagonal descending fillings with row reading $\alpha$ as $\dd(\alpha)$.
\end{definition}

\begin{example} \label{exa:ddfill}
$\dd(212)$ consists of the following four fillings
\[
\begin{array}{|ccc}
2 & \gdot  & \gdot  \\
\gdot  & 1 & \gdot  \\
\gdot  & \gdot  & 2 \\
\hline
2 & 1 & 2
\end{array}
\qquad
\begin{array}{|ccc}
2 & \gdot  & \gdot \\
1 & \gdot   & \gdot \\
\gdot  & 2 & \gdot \\
\hline
3 & 2
\end{array}
\qquad
\begin{array}{|ccc}
\gdot  & \gdot  & 2 \\
\gdot  & 1 & \gdot  \\
2 & \gdot  & \gdot  \\
\hline
2 & 1 & 2
\end{array}
\qquad
\begin{array}{|ccc}
\gdot  & 2 & \gdot \\
1 & \gdot & \gdot \\
2 & \gdot  & \gdot \\
\hline
3 & 2
\end{array}
\]
\end{example}

\begin{definition}
Let $\alpha$ be a composition of $n$, then the \textit{descending quasisymmetric powersum function} is defined as

\begin{equation}\label{eqn:PtoMDD}
    P_\alpha = \sum_{\mathsf{F}\in \dd(\alpha)}M_{\col(\mathsf{F})}.
\end{equation}
\end{definition}

It is clear that these quasipowersums form a basis of QSym as they have a triangular change of basis from the monomials.

\begin{remark}
We briefly mention other families of quasipowersums in Section \ref{sec:faminv}, elsewhere $P_\alpha$ will mean $P_\alpha^\mathcal{D}$.
\end{remark}

Thus from Example \ref{exa:ddfill}, $P_{212} = 2M_{212}+2M_{32}$. It will be convenient to group fillings by row permutation type. First we define orbit representatives for this the group action.

\begin{definition} \label{def:strdia}
A \textit{Strict Diagonal} filling is a filling such that Conditions 1-3 of Definition \ref{def:DD} are satisfied. Denote the set of strict diagonal fillings with $\row(\mathsf{F})=\alpha$ as $\sd(\alpha)$.
\end{definition}
The first two fillings of Example \ref{exa:ddfill} are all the fillings in $\sd(212)$. Let $\alpha$ be a composition of $n$ and $m_i(\alpha)$ denote the number of parts of size $i$. Then $\sigma \in \mathfrak{S}_{m_1(\alpha)}\times \cdots \times \mathfrak{S}_{m_n(\alpha)} = \mathfrak{S}_{m(\alpha)}$ is a \textit{row permutation} and acts on fillings by permuting two rows that contain parts of like sizes. Let $\mathfrak{S}_\mathsf{F}$ be the coset of $\mathfrak{S}_{m(\alpha)}$ containing $\sigma$ such that for any two entries $a_i = a_j$, then $\sigma(j) = k$ implies that $a_j$ and $a_k$ are not in the same column. Note that $\col(\sigma\mathsf{F})=\col(\mathsf{F})$ and $\mathfrak{S}_{\mathsf{F}}$ is dependent only on $\row(\mathsf{F})$ and $\col(\mathsf{F})$.

Notice that we can define a $\dd$ filling as a pair of $\sd$ filling and a row permutation, which leads us to an alternate definition,
\begin{equation} \label{eqn:PtoMSD}
    P_\alpha =\sum_{\substack{\mathsf{F}\in \sd(\alpha)\\ \sigma \in \mathfrak{S}_{\mathsf{F}}}}M_{\col(\sigma(\mathsf{F}))}=\sum_{\mathsf{F}\in\sd(\alpha)}|\mathfrak{S}_\mathsf{F}|M_{\col(\mathsf{F})},
\end{equation}
where the last equality comes from the fact that $\col(\sigma\mathsf{F})=\col(\mathsf{F})$.

In \cite{aliniaeifard2021p} the authors independently defined a powersum basis using weighted $P$-partitions, that, when translated to fillings, is precisely this one: define $\pus$ to be a map of fillings by a sorting of rows such that the reading word is a partition. Then the image of $\dd(\alpha)$ with column sum $\beta$ under $\pus$ is the set of fillings $\mathfrak{R}_{\alpha\beta}$ as described in \cite{aliniaeifard2021p}. They also prove the following result.
\begin{theorem} \label{thm:refine}
The quasisymmetric powersum functions refine the symmetric powersum functions. In other words,
\begin{equation} \label{eqn:refine}
p_\lambda=\sum_{\alpha:\sort(\alpha)=\lambda}P_\alpha
\end{equation}
\end{theorem}
We remark that one can prove this refinement property by changing $p_\lambda$ and $P_\alpha$ into the quasimonomial basis and showing that the two expressions are equivalent. An analogue of this proof is in Section \ref{sec:NCQrefine}.

\subsection{Subposet Interpretation}
Notice that we can define these quasipowersums by subposets as all fillings only coarsen according to the row reading. Given compositions $\alpha = (a_1,\ldots,a_k)\leq \beta = (b_1,\ldots,b_l)$, there exists $1\leq k_1\leq\dotsc\leq k_j\leq k$ such that
$$b_1 = \sum_{i=1}^{k_1}a_i, \ b_2 = \sum_{i=k_1+1}^{k_2}a_i, \ \ldots ,\ b_l = \sum_{i=k_j+1}^{k}a_i.$$
Given a positive integer $L$, define a composition $c_L(\alpha,\beta)$ as 
$$
c_L(\alpha,\beta) = \left( 
\sum_{i=1}^{k_1}\delta_{La_i}, \  \sum_{i=k_1+1}^{k_2}\delta_{La_i}, \ \ldots \ , \sum_{i=k_j+1}^{k}\delta_{La_i} \right)
$$
where $\delta_{ij}$ is the Kronecker delta. Let $c_L(\alpha,\beta)!$ be the product of the factorials of the parts of $c_L(\alpha,\beta)$.

\begin{example}
Let $\alpha =(3,2,1;1;3;1;1,1;2,1) $ and $\beta = (6,1,3,1,2,3)$. Then 
\begin{align*}
    c_1(\alpha,\beta) &= (0+0+1,1,0,1,1+1,0+1) = (1,1,1,2,1) \\
    c_2(\alpha,\beta) &= (0+1+0,0,0,0,0+0,1+0) = (1,1)\\
    c_3(\alpha,\beta) &= (1+0+0,0,1,0,0+0,0+0) = (1,1)\\
    c_1(\alpha,\beta)!&= 1!\times1!\times1!\times2!\times1! = 2.
\end{align*}

\end{example}

The following proposition is needed in proving formulas for the product, coproduct, and that the dual basis is the Zassenhaus basis which are Theorems \ref{thm:coprod}, \ref{thm:Qprod}, and \ref{thm:dual} respectively.

\begin{proposition}\label{prp:coefcat}
Let $\alpha,\alpha',\beta,\beta'$ be compositions such that $\alpha\leq\alpha'$ and $\beta\leq\beta'$. Then 
\[
c_i(\alpha|\beta,\alpha'|\beta)=c_i(\alpha|\alpha')|c_i(\beta,\beta')
\]
where $\alpha|\alpha'$ denotes the concatenation of $\alpha$ and $\alpha'$.
\end{proposition}
This proposition follows from the definition and the example.

Let $\mathcal{D}$ be the subposet of $\mathcal{P}$ with the cover relation $\lessdot_\mathcal{D}$ given by: $\alpha\lessdot_\mathcal{D}\beta$ if $\beta$ is obtained from $\alpha$ by summing two descending adjacent parts, i.e. $\beta = (a_1,\dotsc,a_i+a_{i+1},\dotsc,a_k)$ for some $i$ satisfying $a_i\geq a_{i+1}$. Then define $C(\alpha)$ be the maximal composition of the subposet $\mathcal{D}$ containing $\alpha$. From Example 1, $C(3211311121)=763$.

Let $\mathsf{F}$ be a $\sd$ filling, then $|\mathfrak{S}_\mathsf{F}|=C_{\row(\mathsf{F})\col(\mathsf{F})}$. The number of permutations of rows with entry $i$ is $m_i(\row(\mathsf{F}))!$, then divide by the number of permutations that only permute $a_j$ and $a_k$ if they are in the same column and $a_j=a_k=i$, which is $c_i(\row(\mathsf{F}),\col(\mathsf{F}))!$. This leads to an alternate definition of a descending quasipowersum,

\begin{definition}\label{def:pwrsm}
Let $\alpha$ and $\beta$ be partitions of $n$ and 
\begin{equation}\label{eqn:coef}
C_{\alpha\beta}=\prod^m_{i=1}\frac{m_i(\alpha)!}{c_i(\alpha,\beta)!}.    
\end{equation}

Then an alternate definition of the \textit{(descending) quasisymmetric powersum} function is 
\begin{equation} 
P_\alpha = \sum_{\alpha\leq\beta\leq C(\alpha)}C_{\alpha\beta}M_\beta= \sum_{\alpha\leq_\mathcal{D}\beta}C_{\alpha\beta}M_\beta.
\end{equation}
\end{definition}

\begin{example}
\begin{align*}
P_{1211} &= \left(\frac{3!}{1!1!1!}\right)\left(\frac{1!}{1!}\right)M_{1211} +\left(\frac{3!}{1!2!}\right)\left(\frac{1!}{1!}\right)M_{122} \\
& \qquad \qquad +\left(\frac{3!}{1!1!1!}\right)\left(\frac{1!}{1!}\right)M_{131} +\left(\frac{3!}{1!2!}\right)\left(\frac{1!}{1!}\right)M_{14}\\ 
&= 6M_{1211} + 3M_{122}+6M_{131}+3M_{14}
\end{align*}
\end{example}

\subsection{Product and Coproduct} \label{sec:prodco}
Fix an $\sd$ filling $\mathsf{F}$, a \textit{deconcatenation of fillings}, denoted $\mathsf{F}=\mathsf{F_1}|\mathsf{F_2}$, is achieved by drawing a vertical line in between two columns to get two $\sd$ fillings $\mathsf{F}_1$ and $\mathsf{F}_2$. Notice that (for $\sd$ fillings) there is also a horizontal line that we could draw to deconcatenate the filling. Moreover, $\col(\mathsf{F}_1)|\col(\mathsf{F}_2)=\col(\mathsf{F})$ and $\row(\mathsf{F}_1)|\row(\mathsf{F}_2)=\row(\mathsf{F})$.
\begin{example} The deconcatenation of a filling mimics the deconcatenation of a quasimonomial.
\[
\begin{array}{c|::cccc}
\hdashline
\hdashline
2 & 2      & \gdot  & \gdot & \gdot  \\
2 & 2      & \gdot  & \gdot & \gdot  \\
1 & \gdot  & 1      & \gdot & \gdot  \\
3 & \gdot  & \gdot  & 3     & \gdot  \\
\hline
& 4 & 1 & 3 &
\end{array}
\quad
\begin{array}{c|c::ccc}
2 & 2      & \gdot  & \gdot & \gdot  \\
2 & 2      & \gdot  & \gdot & \gdot  \\
\hdashline
\hdashline
1 & \gdot  & 1      & \gdot & \gdot  \\
3 & \gdot  & \gdot  & 3     & \gdot  \\
\hline
& 4 & 1 & 3 &
\end{array}
\quad
\begin{array}{c|cc::cc}
2 & 2      & \gdot  & \gdot & \gdot  \\
2 & 2      & \gdot  & \gdot & \gdot  \\
1 & \gdot  & 1      & \gdot & \gdot  \\
\hdashline
\hdashline
3 & \gdot  & \gdot  & 3     & \gdot  \\
\hline
& 4 & 1 & 3 &
\end{array}
\quad
\begin{array}{c|ccc::c}
2 & 2      & \gdot  & \gdot & \gdot  \\
2 & 2      & \gdot  & \gdot & \gdot  \\
1 & \gdot  & 1      & \gdot & \gdot  \\
3 & \gdot  & \gdot  & 3     & \gdot  \\
\hdashline
\hdashline
\hline
& 4 & 1 & 3 &
\end{array}
\]
Likewise, $\Delta(M_{413})=1 \otimes M_{413} + M_{4}\otimes M_{13} + M_{41}\otimes M_{3} + M_{413}\otimes 1$. 
\end{example}

Evidently the coproduct of a quasimonomial function can be expressed as 
\begin{equation} \label{eqn:Mcoprod}
    \Delta(M_{\col(\mathsf{F})})=\sum_{\mathsf{F}_1|\mathsf{F}_2=\mathsf{F}}M_{\col(\mathsf{F}_1)}\otimes M_{\col(\mathsf{F}_2)},
\end{equation}
where $\mathsf{F}$ is an $\sd$ filling.

\begin{theorem} \label{thm:coprod}
Let $\gamma$ be a composition of $n$, then a quasisymmetric powersum has a deconcatenate coproduct:
\begin{equation} \label{eq:coprod}
\Delta(P_\gamma)=\sum_{\alpha|\beta=\gamma}\prod_{i=1}^n\frac{m_i(\gamma)!}{m_i(\alpha)!m_i(\beta)!}P_\alpha\otimes P_\beta.
\end{equation}
\end{theorem}

\begin{proof}
By \eqref{eqn:PtoMSD} and \eqref{eqn:Mcoprod}
\begin{equation*}
    \Delta(P_\gamma)=\Delta\left(\sum_{\mathsf{F}\in\sd(\gamma)}|\mathfrak{S}_\mathsf{F}|M_{\col(\mathsf{F})}\right)=\sum_{\mathsf{F}\in\sd(\gamma)}\sum_{\mathsf{F}_1|\mathsf{F}_2=\mathsf{F}}|\mathfrak{S}_\mathsf{F}|M_{\col(\mathsf{F}_1)}\otimes M_{\col(\mathsf{F}_2)}.
\end{equation*}

Let $m$ be the degree of $\col(\mathsf{F}_1)$, then fix $m$. Let $\mathsf{F}_1'$ be the diagonal filling with $\row(\mathsf{F}_1) = \row(\mathsf{F}_1')$, say $\row(\mathsf{F}_1')=(\gamma_1,\ldots,\gamma_j)$. Then every filling of degree $m$ comes from moving entry $\gamma_{i+1}$ to the column of $\gamma_i$ when $\gamma_i\geq\gamma_{i+1}$ for certain choices of $i$. In other words, the fillings of degree $m$ is precisely $\sd(\gamma_1,\ldots,\gamma_j)$. The same is said for the right side, and thus 

\begin{equation*}
    \sum_{\mathsf{F}\in\sd(\gamma)}\sum_{\mathsf{F}_1|\mathsf{F}_2=\mathsf{F}}|\mathfrak{S}_\mathsf{F}|M_{\col(\mathsf{F}_1)}\otimes M_{\col(\mathsf{F}_2)}=\sum_{\alpha|\beta=\gamma}\sum_{\substack{\mathsf{F}_1\in\sd(\alpha)\\\mathsf{F}_2\in\sd(\beta)}}\left|\mathfrak{S}_{\mathsf{F}_1|\mathsf{F}_2}\right|M_{\col(\mathsf{F}_1)}\otimes M_{\col(\mathsf{F}_2)}.
\end{equation*}
Finally to express this in the quasipowersum basis we divide the coefficient by $|\mathfrak{S}_{\mathsf{F}_1}||\mathfrak{S}_{\mathsf{F}_2}|$. We simplify the coefficient through \eqref{eqn:coef} and Proposition \ref{prp:coefcat}
\begin{equation}
    \frac{|\mathfrak{S}_{\mathsf{F}_1|\mathsf{F}_2}|}{|\mathfrak{S}_{\mathsf{F}_1}||\mathfrak{S}_{\mathsf{F}_2}|}=\prod_{i=1}^n\frac{\frac{m_i(\gamma)!}{c_i(\gamma,\col(\mathsf{F}_1|\mathsf{F}_2) )!}}{\frac{m_i(\alpha)!}{c_i(\alpha,\col(\mathsf{F}_1))!}\frac{m_i(\beta)!}{c_i(\beta,\col(\mathsf{F}_2))!}} =\prod_{i=1}^n\frac{m_i(\gamma)!}{m_i(\alpha)!m_i(\beta)!}
\end{equation}
The coefficient is interpreted as all the permutations of $\mathfrak{S}_\mathsf{F}$ that permuted a row in $\mathsf{F}_1$ to a row in $\mathsf{F}_2$.
\end{proof}

The \textit{shuffle}, denoted as $\shuffle$, of two compositions $\alpha = (a_1,\ldots,a_j)$ and $\beta = (b_1,\ldots,b_k)$ is the multiset of compositions defined recursively by 
\begin{enumerate}
    \item $\emptyset\shuffle \alpha = \alpha \shuffle \emptyset = \{\alpha\}$
    \item $\alpha\shuffle\beta = a_1|\big((a_2,\ldots,a_j)\shuffle \beta\big) + b_1|\big( \alpha\shuffle (b_2,\ldots,b_k) \big)$,  
\end{enumerate}
using $+$ for disjoint (multiset) union. If we add to the recursion condition 2) the term $(a_1+b_1)|\big( (a_2,\ldots,a_j) \shuffle (b_2,\ldots,b_k) \big)$, then this is called the \textit{quasi-shuffle} and is denoted as $\qshuf$. Recall that the product rule of the quasimonomial basis is a quasishuffle product:
\[
M_\alpha M_\beta = \sum_{\gamma \in \alpha \qshuf \beta}M_\gamma.
\]

Likewise we define the quasishuffle of fillings $\mathsf{F_1} = (a_1,\ldots,a_j)$ and $\mathsf{F_2}= (b_1,\ldots,b_k)$, where $a_i$ and $b_i$ are the columns of $\mathsf{F_1}$ and $\mathsf{F_2}$ respectively, is defined recursively
\begin{enumerate}
    \item $\emptyset\shuffle \mathsf{F_1} = \mathsf{F_1} \shuffle \emptyset = \{\mathsf{F_1}\}$ where $\emptyset$ is the empty filling
    \item $\mathsf{F_1}\shuffle\mathsf{F_2} = a_1|\big((a_2,\ldots,a_j)\shuffle \mathsf{F_2}\big) + b_1|\big( \mathsf{F_1}\shuffle (b_2,\ldots,b_k) \big) + (a_1+b_1)|\big( (a_2,\ldots,a_j) \shuffle (b_2,\ldots,b_k) \big)$  
\end{enumerate}
where: concatenating columns $a_1$ and $a_2$ means starting $a_2$ southeast of $a_i$; and addition $a_1+a_2$ builds one column by sort combining the two columns.

\begin{example}
The following illustrates that $M_{23}M_1=M_{231}+M_{213}+M_{24}+M_{123}+M_{33}$.
\[
\begin{array}{c|cc}
2 & 2      & \gdot \\
1 & \gdot  & 1     \\
\hline
& 2 & 1
\end{array}
\,
\qshuf
\,
\begin{array}{c|c}
3 & 3 \\
\hline
  & 3 
\end{array}
=
\begin{array}{c|c}
2 & 2  \\
\hline
  & 2 
\end{array}
\Bigg|\left(
\begin{array}{c|c}
1 & 1  \\
\hline
  & 1 
\end{array}
\:
\qshuf
\:
\begin{array}{c|c}
3 & 3 \\
\hline
  & 3 
\end{array}
\right)
+ \,
\begin{array}{c|c}
3 & 3 \\
\hline
  & 3 
\end{array}
\Bigg| \left(
\begin{array}{c|cc}
2 & 2      & \gdot \\
1 & \gdot  & 1     \\
\hline
& 2 & 1
\end{array} 
\: \qshuf \: \emptyset
\right)
+
\begin{array}{c|cc}
3 & 3  \\
2 & 2  \\
\hline
  & 5
\end{array}
\Bigg| \left(
\begin{array}{c|c}
1 & 1  \\
\hline
  & 1 
\end{array}
\: \qshuf  \: \emptyset
\right)
\]
\[
=
\begin{array}{c|cc}
2 & 2 & \gdot \\
1 & \gdot & 1  \\
\hline
  & 2 &1
\end{array}
\Bigg|\left(
\emptyset \: \qshuf \:
\begin{array}{c|c}
3 & 3 \\
\hline
  & 3 
\end{array}
\right) +
\begin{array}{c|cc}
2 & 2 & \gdot \\
3 & \gdot & 3  \\
\hline
  & 2 &3
\end{array}
\Bigg| \left(
\begin{array}{c|c}
1 & 1  \\
\hline
  & 1 
\end{array}
\: \qshuf  \: \emptyset
\right)
+ 
\begin{array}{c|cc}
2 & 2 & \gdot \\
3 & \gdot & 3  \\
1 & \gdot & 1 \\
\hline
  & 2 & 4
\end{array}
+
\begin{array}{c|ccc}
3 & 3 & \gdot & \gdot \\
2 & \gdot & 2 & \gdot \\
1 & \gdot & \gdot & 1 \\
\hline
  & 3 & 2 & 1
\end{array}
+
\begin{array}{c|cc}
3 & 3 & \gdot   \\
2 & 2 & \gdot \\
1 & \gdot & 1 \\
\hline
  & 5 & 1
\end{array}
\]
\[
=
\begin{array}{c|ccc}
2 & 2 & \gdot & \gdot \\
1 & \gdot & 1 & \gdot  \\
3 & \gdot & \gdot & 3 \\
\hline
  & 2 & 1 & 3
\end{array}
+ 
\begin{array}{c|ccc}
2 & 2 & \gdot & \gdot \\
3 & \gdot & 3 & \gdot \\
1 & \gdot & \gdot & 1 \\
\hline
  & 2 & 3 & 1
\end{array}
+ 
\begin{array}{c|cc}
2 & 2 & \gdot \\
3 & \gdot & 3  \\
1 & \gdot & 1 \\
\hline
  & 2 & 4
\end{array}
+
\begin{array}{c|ccc}
3 & 3 & \gdot & \gdot \\
2 & \gdot & 2 & \gdot \\
1 & \gdot & \gdot & 1 \\
\hline
  & 3 & 2 & 1
\end{array}
+
\begin{array}{c|cc}
3 & 3 & \gdot   \\
2 & 2 & \gdot \\
1 & \gdot & 1 \\
\hline
  & 5 & 1
\end{array}
\]
\end{example}
Observe from the last example that if $\mathsf{F}\in\mathsf{F}_1\qshuf\mathsf{F}_2$, then $\row(\mathsf{F})\in\row(\mathsf{F}_1)\shuffle\row(\mathsf{F}_2)$ and $\col(\mathsf{F})\in\col(\mathsf{F}_1)\qshuf\col(\mathsf{F}_2)$.

\begin{theorem}\label{thm:Qprod}
Let $\alpha$ and $\beta$ be compositions of $n$, then the quasipowersums have a shuffle product, in other words

\[
P_\alpha P_\beta = \sum_{\gamma \in \alpha\shuffle\beta} \prod_i \frac{m_i(\alpha)!m_i(\beta)!}{m_i(\gamma)!} P_\gamma.
\]
\end{theorem}
\begin{proof}
From \eqref{eqn:PtoMSD}

\begin{align*}
    P_\alpha  P_\beta &= \left( \sum_{\mathsf{F}_1\in\sd(\alpha)}|\mathfrak{S}_{\mathsf{F}_1}| M_{\col(\mathsf{F}_1)} \right)  \left( \sum_{\mathsf{F}_2\in\sd(\beta)}|\mathfrak{S}_{\mathsf{F}_2}| M_{\col(\mathsf{F}_2)} \right) \\
    &= \sum_{\substack{\mathsf{F}_1\in \sd(\alpha),\mathsf{F}_2\in\sd(\beta)\\\mathsf{F}\in\mathsf{F}_1\qshuf\mathsf{F}_2}} |\mathfrak{S}_{\mathsf{F}_1}||\mathfrak{S}_{\mathsf{F}_2}| M_{\col(F)} .
\end{align*}

In order to express the above in the powersum basis, we have to switch the sums, i.e. use $\sd(\alpha\shuffle\beta)=\bigcup_{\gamma\in\alpha\shuffle\beta}\sd(\gamma)$ instead of $\sd(\alpha)\qshuf\sd(\beta)$. We will first show that these two are equivalent as sets, then show that multiplying by $O_\mathsf{F}$ (to be defined) will result in them being equal as multisets. Finally, we will reduce the coefficient.

Notice that any filling in $\sd(\alpha)\qshuf\sd(\beta)$ is also in $\sd(\alpha\shuffle\beta)$ due to the fact that every filling $\mathsf{F}\in \mathsf{F}_1\qshuf \mathsf{F}_2$ has the property that $\row(\mathsf{F})\in\row(\mathsf{F}_1)\shuffle\row(\mathsf{F}_2)$. Conversely, we can also construct any filling $\mathsf{F}\in\sd(\alpha\shuffle\beta)$ as a quasishuffle of two fillings $\mathsf{F}_1$ and $\mathsf{F}_2$ of $\sd(\alpha)$ and $\sd(\beta)$. Fix a filling $\mathsf{F}$, then if the first row is from $\alpha$, make a new filling $\mathsf{F}_1'$ where the first row is $\mathsf{F}_1'$ is the first row of $\mathsf{F}$ (and likewise for $\beta$) and remove the first row from $\mathsf{F}$. Continue to do this for all row to get two fillings $\mathsf{F}_1$ and $\mathsf{F}_2$ from $\sd(\alpha)$ and $\sd(\beta)$. Thus $\sd(\alpha)\qshuf\sd(\beta)$ and $\sd(\alpha\shuffle\beta)$ are equivalent as sets, but not as multisets.

In order to switch the sum from  $\sd(\alpha)\qshuf\sd(\beta)$ to $\sd(\alpha\shuffle\beta)$, we must consider the double count that comes from the quasishuffle, which will be denoted as $O_{\mathsf{F}}$. We illustrate this with an example

\begin{example}
Let $\mathsf{F}_1\in\sd(2,1,1,1)$ and $\mathsf{F}_2\in\sd(1,1,3)$ be the fillings shown in Figure \ref{fig:fillings}. One of the resulting fillings is when we combine the second column of $\mathsf{F}_1$ and the first column of $\mathsf{F}_2$. Notice that there is only one way to make this filling. However if we consider the number of ways that this filling is generated by $\sd((2,1,1,1)\shuffle(1,1,3))$, then that would be $\binom{5}{3}$.

\begin{figure}[h]
    \centering
    \[ \footnotesize
    \begin{array}{c|cc}
    2 & 2 & \gdot \\
    1 & \gdot & 1   \\
    1 & \gdot & 1 \\
    1 & \gdot & 1 \\
    \hline
    & 2 & 3
    \end{array}
    \quad
    \qshuf
    \quad 
    \begin{array}{c|cc}
    1 &  1 & \gdot  \\
    1 &  1 & \gdot \\
    3 & \gdot & 3 \\
    \hline
    & 2 & 3
    \end{array}
    \quad
    \ni
    \quad
    \begin{array}{c|ccc}
    2 & 2 & \gdot & \gdot \\
    1 & \gdot & 1 & \gdot  \\
    1 & \gdot & 1 & \gdot \\
    1 & \gdot & 1 & \gdot \\
    1 & \gdot & 1 & \gdot  \\
    1 & \gdot & 1 & \gdot \\
    3 & \gdot & \gdot & 3 \\
    \hline
    & 2 & 5 & 3
    \end{array}
    \]
    \caption{The quasishuffle of fillings}
    \label{fig:fillings}
\end{figure}
\end{example}

Let $\alpha = (a_1,\ldots,a_l)$, $\beta = (b_1,\ldots,b_{l'})$, and fix $\mathsf{F}\in\sd(\alpha\shuffle\beta)$, $\mathsf{F}_1\in\sd(\alpha)$, and $\mathsf{F}_2\in\sd(\beta)$ such that $\mathsf{F}\in \mathsf{F}_1\qshuf\mathsf{F}_2$. Then we define the partitions $\alpha_j$ and $\beta_{j'}$ to be the entries of column $j$ of $\mathsf{F}$ and column $j'$ of $\mathsf{F}_2$ respectively. Then it is easy to see that $\col(\mathsf{F}_1)=(|\alpha_1|,\ldots,|\alpha_k|)$ and $\alpha = \alpha_1|\cdots|\alpha_k$. Let $J$ be a column of $\mathsf{F}$ such that the combination of columns $j$ and $j'$ from fillings $\mathsf{F}_1$ and $\mathsf{F}_2$ results in the column $J$. The number of ways build the column $J$ of $\mathsf{F}$ from the shuffle of $\alpha_j$ and $\beta_{j'}$ is $\binom{m_i(\alpha_j)+m_i(\beta_{j'})}{m_i(\alpha_j)}$. Then we define $O_\mathsf{F}$ as the reciprocal of all such ways to build $J$ for all columns with such $J$. To switch the sum as described above, we multiply $O_{\mathsf{F}}$ such that for any two columns that combine in result of the quasishuffle, say $\alpha_j$ and $\beta_{j'}$, then 
\[
O_{\mathsf{F}} =\prod_{i}\frac{1}{\prod_{j,j'}\binom{m_i(\alpha_j)+m_i(\beta_{j'})}{m_i(\alpha_j)}}
\]
where the product runs over all $j$, $j'$ where $\alpha_j$ is combined with $\beta_{j'}$. From the example above, $O_\mathsf{F}= 1/\binom{3+2}{3}$.

Finally the coefficient that is obtained from changing the expression into the powersum basis is as follows
\[ 
O_{\mathsf{F}}\frac{|\mathfrak{S}_{\mathsf{F}_1}||\mathfrak{S}_{\mathsf{F}_2}|}{|\mathfrak{S}_\mathsf{F}|}
=\prod_{i=1}^n \frac{1}{\prod_{j,j'} \binom{m_i(\alpha_j)+m_i(\beta_{j'})}{m_i(\alpha_j)}} 
\frac{c_i(\gamma,\col(\mathsf{F}_1\qshuf\mathsf{F}_2) )!}{c_i(\alpha,\col(\mathsf{F}_1))!c_i(\beta,\col(\mathsf{F}_2))!}
\left[\frac{m_i(\alpha)!m_i(\beta)!}{m_i(\gamma)!}\right]
\]
It suffices to show that this expression reduces to what is inside the bracket. By Proposition \ref{prp:coefcat} we can write $c_i((a_1,\ldots,a_l),(|\alpha_1|,\ldots,|\alpha_k|))!=c_i((a_1,\ldots,a_{i_1}),|\alpha_1|)! \cdots c_i((a_{i_{k-1}+1},\ldots,a_{l}),|\alpha_k|)!$, and similarly for $\beta$. Consequently everything outside of the bracket can be considered column by column, while what is inside the bracket comes from the whole of a filling. Since $\mathsf{F}$ is a quasishuffle of $\mathsf{F}_1$ and $\mathsf{F}_2$ there are three ways a column of $\mathsf{F}$ is made.

First consider the case when a column $a_j$ isn't combined. Fix $i$, then $O_\mathsf{F}$ doesn't contribute to the coefficient which leaves us with 
\[
\frac{c_i((a_{i_{j-1}+1},\ldots,a_j),|\alpha_j|)!}{c_i((a_{i_{j-1}+1},\ldots,a_j),|\alpha_j|)!}
\]
which is obviously 1. The same argument is made when $\beta_{j'}$ isn't combined with any other column.

The third case is when two columns $\alpha_j$ and $\beta_j'$ combine. Fix $i$, and we get the expression
\[
\frac{1}{\binom{m_i(\alpha_j)+m_i(\beta_{j'})}{m_i(\beta_{j'})}}\frac{c_i((a_{i_{j-1}+1},\ldots,a_j)\shuffle(b_{i_{{j'}-1}+1},\ldots,b_{j'}),|\alpha_j|+|\beta_{j'}|)!}{c_i((a_{i_{j-1}+1},\ldots,a_j),|\alpha_j|)!c_i((b_{i_{{j'}-1}+1},\ldots,b_{j'}),|\beta_{j'}|)!}
\]

Recall that the formula of $c_i$ counts the number of $i$'s in the composition $(a_{i_{j-1}+1},\ldots,a_j)|(b_{i_{{j'}-1}+1},\ldots,b_{j'})$, which is, in this case, the same as the function $m_i$. Thus using this information and writing out the choose function yields
\[
\frac{m_i(\alpha_{j})!(m_i(\alpha_j)+m_i(\beta_{j'})-m_i(\alpha_j))!}{(m_i(\alpha_j)+m_i(\beta_{j'}))!}\frac{m_i((a_{i_{j-1}+1},\ldots,a_j)\shuffle(b_{i_{{j'}-1}+1},\ldots,b_{j'}))!}{m_i((a_{i_{j-1}+1},\ldots,a_j))!m_i((b_{i_{{j'}-1}+1},\ldots,b_{j'})))!}.
\]
Recall that $\alpha_j$ are the entries of column $j$ and $(a_{i_{j-1}+1},\ldots,a_j)$ are also the entries of column $j$, thus the expression is 1, which completes the proof.
\end{proof}

We with hold a formula for antipode for now.

\begin{remark}
Theorems \ref{thm:coprod} and \ref{thm:Qprod} can be proven (perhaps more straightforwardly) without using fillings. However, our chosen approach, reveals a Hopf algebra structure on fillings. We have defined a product and a coproduct that satisfies the conditions for being a Hopf algebra and because this is graded connected, there exists an antipode.
\end{remark}

\section{The Scaled Quasi-Powersum Basis and its Dual}\label{sec:scapwr}
\subsection{Scaled Quasi-Powersum Basis}
We will now look at a scaled version of this basis that will end up being the dual of the Zassenhauss basis and will have a nice product and coproduct. Recall the scalar
\[
z_\lambda = \prod_ii^{m_i(\lambda)}m_i(\lambda)!.
\]
We extend this definition to compositions by $z_\alpha=z_{\sort(\alpha)}$. Symmetric powersums have the property that they form an orthogonal basis for $\sym$, to be more exact, $\langle p_\lambda,p_\mu\rangle=z_\lambda\delta_{\lambda\mu}$ where $\langle,\rangle$ is the Hall inner product. Thus the \textit{scaled symmetric powersums} is defined as $\widetilde{p}_\lambda=z_\lambda^{-1} p_\lambda$, so that $\langle \widetilde{p}_\lambda,p_\mu\rangle=\delta_{\lambda\mu}$. Though $\qsym$ is not self-dual, we define scaled quasipowersums in the same way.

\begin{definition}
A Scaled Quasi-Powersum functions are defined as 
\begin{equation} \label{def:scaP}
\tilde P_{\alpha} = \frac{1}{z_\alpha} P_{\alpha}. 
\end{equation}

\end{definition}

By multiplying both sides of \eqref{eqn:refine} by $z_\lambda^{-1}$ it is evident that the scaled quasipowersum functions refines $\widetilde{p}_\lambda$. However with this new scaling there is a nicer product and coproduct.

\begin{corollary}
The scaled quasipowersum functions has a deconcatenate coproduct and a shuffle product, in other words
\[
\Delta(\widetilde{P}_\gamma)=\sum_{\alpha|\beta=\gamma}\widetilde{P}_\alpha\otimes \widetilde{P}_\beta
\qquad \text{and} \qquad \widetilde{P}_\alpha\otimes\widetilde{P}_\beta=\sum_{\gamma\in \alpha\shuffle\beta}\widetilde{P}_\gamma.
\]
\end{corollary}
\begin{proof}
Substituting \eqref{def:scaP} into \eqref{eq:coprod} yields 
\[
\Delta(z_{\gamma}\widetilde{P}_\gamma)=\sum_{\alpha|\beta=\gamma}\prod_{i=1}^n\frac{m_i(\gamma)!}{m_i(\alpha)!m_i(\beta)!}z_\alpha \widetilde{P}_\alpha\otimes z_\beta \widetilde{P}_\beta.
\]
The coefficient reduces to
\[
\prod_{i=1}^n\frac{m_i(\gamma)!}{m_i(\alpha)!m_i(\beta)!}\frac{z_\alpha z_\beta}{z_\gamma} =\prod_{i=1}^n\frac{i^{m_i(\alpha)+m_i(\beta)}}{i^{m_i(\gamma)}}=1.
\]
Likewise we do the same for the coefficient of the shuffle product.
\end{proof}

Finally, we look at the antipode of QSym, denoted $S$. Since $\widetilde{P}$ has shuffle product and deconcatenate coproduct, $\widetilde{P}$ makes QSym a shuffle algebra. By Theorem 3.1 of \cite{benedetti2017antipodes} the antipode acting on $\widetilde{P}$ is defined as:

\begin{corollary}\label{cor:antipode}
The antipode of QSym acts on a quasisymmetric powersum function as
\begin{equation}
    S(\widetilde{P}_\alpha)=(-1)^{\len(\alpha)}\widetilde{P}_{ \overleftarrow{\alpha}}
\end{equation}
where $\overleftarrow{\alpha}$ is the reverse composition.

\end{corollary} 
Moreover the previous Corollary implies that $S(P_\alpha)=(-1)^{\len(\alpha)}P_{\overleftarrow{\alpha}}$.

\subsection{Zassenhaus Basis} \label{ssec:Zbasis}
The dual of the quasimonomial basis is the complete basis, denoted $\mathbf{S}_\alpha$. The complete basis and the Zassenhaus basis, denoted $\mathbf{Z}_\alpha$, are both multiplicative and in \cite{krob1997noncommutative} the Zassenhaus basis is defined as 
\begin{equation} \label{eqn:expZ}
\sum_{k}t^k\mathbf{S}_k=\ldots\exp{\frac{\mathbf{Z}_i}{i}t^i}\ldots\exp{\frac{\mathbf{Z}_2}{2}t^2}\exp{\frac{\mathbf{Z}_1}{1}t^1}
\end{equation}

It would be in our interest to rewrite this definition using posets. Thus expand \eqref{eqn:expZ} as
\begin{align*}
    \sum_{k}t^k\mathbf{S}_k&=\cdots\left(1+\frac{\mathbf{Z}_i}{i}t^i+\frac{\left(\frac{\mathbf{Z}_i}{i}t^i\right)^2}{2!}+\cdots \right)\cdots\\ 
    &\qquad\qquad\qquad\ldots\left(1+\frac{\mathbf{Z}_2}{2}t^2+\frac{\left(\frac{\mathbf{Z}_2}{2}t^2\right)^2}{2!}+\cdots \right)\left(1+\frac{\mathbf{Z}_1}{1}t^n+\frac{\left(\frac{\mathbf{Z}_1}{1}t^1\right)^2}{2!}+\cdots \right).
\end{align*}

In the change of basis from $\mathbf{S}_k$ to the Zassenhaus basis every term $\beta$ has the property that $\beta\leq_\mathcal{D}(k)$. For a fixed $\beta$, all parts of size $i$ come from the factor $\exp{\frac{\mathbf{Z}_i}{i}t^i}$, Pick a term $j$ in $\exp{\frac{\mathbf{Z}_i}{i}t^i}$, then the coefficient of $Z_{i^j}t^j$ in $\exp{\frac{\mathbf{Z}_i}{i}t^i}$ is $\frac{1}{i^jj!}$. 
Note that $i^j=i^{m_i(\beta)}$ and that $j!$ is the number of parts with size $i$ that is next to each other, and since $\beta$ is a composition of $k$, $j!=c_i(\beta,k)!$. Then \eqref{eqn:expZ} is written as
\[
\mathbf{S}_k = \sum_{\beta\leq_\mathcal{D}(k)}\prod_i\frac{1}{i^{m_i(\beta)}c_i(\beta,k)!}\mathbf{Z}_\beta
\]
Then by Proposition \ref{prp:coefcat} we get an alternate definition.
\begin{equation} \label{eqn:StoZ}
    \mathbf{S}_\alpha=\sum_{\beta\leq_\mathcal{D}\alpha}\prod_{i}\frac{1}{i^{m_i(\beta)}c_i(\beta,\alpha)!}\mathbf{Z}_\beta
\end{equation}

\begin{theorem} \label{thm:dual}
The dual of the scaled quasipowersum basis is the Zassenhaus basis.
\end{theorem}
\begin{proof}
Combining \eqref{def:pwrsm} and \eqref{def:scaP} yields
\begin{equation} 
\widetilde{P}_\alpha = \sum_{\alpha\leq_\mathcal{D}\beta}\prod_i\frac{1}{i^{m_i(\beta)}c_i(\beta,\alpha)!}M_\beta.
\end{equation}
Comparing this equation and \eqref{eqn:StoZ}, and recalling that $\langle \mathbf{S}_\alpha,M_\beta\rangle=\delta_{\alpha\beta}$, it is apparent that $\langle \mathbf{Z}_\alpha,\widetilde{P}_\beta\rangle=\delta_{\alpha\beta}$.
\end{proof}

\section{Change of Basis} \label{sec:CoB}
Recall that a ribbon is a skew ferrers diagram with no $2\times2$ boxes; for this paper we will allow ribbons to be disconnected, however adjacent boxes at least share a corner (if not an edge). There is a natural ribbon to associate to a composition, which we denote $R(\alpha)$:
\[
\alpha = (2,3,1,1,2)
\longleftrightarrow
\ydiagram{2,1+3,3+1,3+1,3+2}
\]

Let $R$ be a ribbon and $j$ a positive integer, then $R-j$ removes the first $j$ boxes. Let $\alpha=(a_1,\ldots,a_k)$ and $\beta=(b_1,\ldots,b_l)$ be compositions of $n$. We build the tuple of ribbons $D(\beta,\alpha)=(R_1,R_2,\ldots,R_n)$ from the following algorithm, where $R_k$ is nonempty if and only if some $a_i = k$. 

\begin{algorithm}
\caption{Descent Ribbons}
\begin{algorithmic}[1]
\State Initialize $\hgt = 0$, $R = R(\beta)$, $D(\beta,\alpha) = (\emptyset,\ldots,\emptyset)$
\For{$i$ in $[1,\ldots,k]$}:
\Comment{Extend ribbon $R_{a_i}$ at it's bottom right end by adding a box in one of three ways}
\If{$a_{i-1}\neq a_i$ \textbf{or} $a_i = 1$}
\State Extend $R_{a_i}$ by adding a box in the southeast position
\EndIf
\If{$a_{i-1}=a_i$ \textbf{and} there exists a $j$ such that $R(b_j,\ldots,b_l) = R$}
\State Extend $R_{a_i}$ by adding a box in the southeast position
\EndIf
\If{$a_{i-1}=a_i$ \textbf{and} there doesn't exist a $j$ such that $R(b_j,\ldots,b_l) = R$}
\State Extend $R_{a_i}$ by adding a box in the east position
\EndIf
\State $\hgt = \hgt - 1 + \hgt(a_i')$ where $a_i'$ has $a_i$ boxes and $a_i'+ A = R$ for some $A$ 
\State $R = R-a_i$ 
\EndFor
\end{algorithmic}
\end{algorithm}

\begin{example}\label{ex:tuprib}
Let $\alpha = (1,2,1,1)$ and $\beta=(1,1,3)$. We find $D(\beta,\alpha)$ using the algorithm. In a given step, the black has been removed from $R$ and the $\bullet$ are about to be removed from $R$. Then to find $D(\beta,\alpha)$:

\begin{align*}
    R:\quad& \ydiagram{1,1,3} & & \ydiagram[\bullet]{1}*{1,1,3} & & \ydiagram[*(black)]{1}*[\bullet]{0,1,1}*{1,1,3} & & &\ydiagram[*(black)]{1,1,1}*[\bullet]{0,0,1+1}*{1,1,3}& & & &\ydiagram[*(black)]{1,1,2}*[\bullet]{0,0,2+1}&\\
    D(\alpha,\beta):\, & (\emptyset,\emptyset) &\longrightarrow& \left(\ydiagram{1},\emptyset \right) &\longrightarrow& \left(\ydiagram{1},\ydiagram{1} \right)&\longrightarrow& &\left(\ydiagram{1,1+1},\ydiagram{1} \right)& &\longrightarrow& &\left(\ydiagram{1,1+2},\ydiagram{1} \right)&\\
    \hgt:\quad& \hgt=0 & & \hgt=0 & & \hgt=1 & & &\hgt=1& & & &\hgt=1& 
\end{align*}
\end{example}

We denote the height in the algorithm as $\hgt(\beta,\alpha)$. 

\begin{definition}
Let $\alpha$ be a composition of $n$. A filling of the ribbon of $\alpha$ is a \textit{Standard Descent Ribbon} if the Ribbon filling is standard (i.e. the numbers 1 to $n$ each appear once) and increasing from left to right, and decreasing from top to bottom.
\end{definition}

Denote the set of all standard descent ribbon fillings of $D(\alpha,\beta)$ as $\sdr(\beta,\alpha)$. Finally, let $\alpha = (a_1,\ldots,a_k)$ be a composition of $n$, then break $\alpha$ at its ascents such that $\alpha=\gamma_1|\cdots|\gamma_l$ where $\gamma_i=(a_{j_1},\ldots,a_{j_1+j_2})$ where $a_{j_1-1}<a_{j_1}\geq\ldots\geq a_{j_1+j_2}<a_{j_1+j_2+1}$. Then $T(\alpha)=\gamma_1'|\cdots|\gamma_l'$.
For example, $I(1,2,1,1)=(1,1,2,1)$ and $C(1,2,1,1)=(1,4)$. Note that, in subset notation, $\alpha=C(\alpha)\cup D(\alpha)$, and $T(\alpha)=C(\alpha) \cup D(\alpha)^C$.

\begin{theorem}\label{thm:MNrule}
Let $\alpha$ be a composition, then 
\begin{equation}
    P_\alpha = \sum_{T(\alpha)\leq\beta\leq C(\alpha)}(-1)^{\hgt(\beta,\alpha)}|\sdr(\beta,\alpha)|F_\beta.
\end{equation}
\end{theorem}
The proof of this theorem is in Section \ref{sec:proof}.

\begin{example}
Continuing from Example \ref{ex:tuprib}, the coefficient of $F_{(1,1,3)}$ when $\alpha = (1,2,1,1)$ is $-3$ since the $\hgt(113,1211)=1$ and the fillings of $\sdr(113,1211)$ are
\[\footnotesize
\left(
\ytableausetup{notabloids}
\begin{ytableau}
1 \\
\none & 2 & 3 \\
\end{ytableau}
\;,\;
\begin{ytableau}
1 \\
\end{ytableau}
\right)\;,\;\left(
\ytableausetup{notabloids}
\begin{ytableau}
2 \\
\none & 1 & 3 \\
\end{ytableau}
\;,\;
\begin{ytableau}
1 \\
\end{ytableau}
\right)\;,\;\left(
\ytableausetup{notabloids}
\begin{ytableau}
3 \\
\none & 1 & 2 \\
\end{ytableau}
\;,\;
\begin{ytableau}
1 \\
\end{ytableau}
\right)
\]
Repeat for compositions in the interval $[T(\alpha),C(\alpha)]=[(1,1,2,1),(1,4)]$ to get $P_{(1,2,1,1)}=-3F_{(1,1,2,1)}-3F_{(1,1,3)}+3F_{(1,3,1)}+3F_{(1,4)}$.
\end{example}

\section{Quasipowersum Families and Involutions} \label{sec:faminv}
Let $\prec$ denote any total ordering on positive integers. Let $\mathcal{R}$ be the subposet of $\mathcal{P}$ with the cover relation $\alpha\lessdot_\prec\beta$ if $\beta$ is obtained from $\alpha$ by summing two descending adjacent parts with respect to $\preceq$, i.e. $\beta = (a_1,\dotsc,a_i+a_{i+1},\dotsc,a_k)$ \textit{only if} $a_i\succeq a_{i+1}$. From above, $\mathcal{D}$ is a special case of $\mathcal{R}$ where $\succeq$ is defined by size. We define $C^\prec(\alpha)$ just as above except with the subposet $\mathcal{R}$. Then we can define a whole family of quasisymmetric powersums as  
\begin{equation} \label{eqn:Rpwrsum}
    P^\prec_\alpha = \sum_{\alpha\leq\beta\leq C^\prec(\alpha)}C_{\alpha\beta}M_\beta= \sum_{\alpha\leq_\prec\beta}C_{\alpha\beta}M_\beta.
\end{equation}

\begin{example}
Consider the total order on positive integers $a_i\succeq_\mathcal{E} a_{i+1}$ if $a_i$ is even and $a_{i+1}$ is odd, or if $a_i$ and $a_{i+1}$ is both even or odd and $a_i\succeq a_{i+1}$. Then define the subposet on compositions $\mathcal{E}$ with the cover relation $\alpha\lessdot_\mathcal{E}\beta$ if $\beta = (a_1,\dotsc,a_i+a_{i+1},\dotsc,a_k)$ \textit{only if} $a_i\succeq_\mathcal{E} a_{i+1}$. Thus we get another type of quasipowersum
\[
P^\mathcal{E}_{23426} = 2M_{23426} + 2M_{5426} + 2M_{2366} + 2M_{566}.
\]
\end{example}

Similarly, we can adapt these quasipowersums to be generated by fillings by switching $\leq$ to $\preceq$ in Definition \ref{def:strdia}. Thus these types of quasipowersums have the property that it refines the symmetric powersums, has a shuffle product, and has a deconcatenate coproduct, and the dual up to scaling is of the form
\[
    \sum_{k}t^k\mathbf{S}_k=\ldots\exp{\frac{\mathbf{Z}_{n_i}^\prec}{n_i}t^{n_i}}\ldots\exp{\frac{\mathbf{Z}_{n_2}^\prec}{n_2}t^{n_2}}\exp{\frac{\mathbf{Z}_{n_1}^\prec}{n_1}t^{n_1}}
\]
where $N = \{\ldots,n_i,\ldots,n_2,n_1 \}$ where $n_i>_\prec n_{i-1}$ and $n_i\in \mathbb{N}$, simply by switching $\leq$ to $\preceq$ in relevant proofs. For example, $\mathbf{S}_{23426} = 2\mathbf{Z}^{\mathcal{E}}_{23426} + 2\mathbf{Z}^{\mathcal{E}}_{5426} + 2\mathbf{Z}^{\mathcal{E}}_{2366} + 2\mathbf{Z}^{\mathcal{E}}_{566}$ where $\langle P^{\mathcal{E}}_\alpha,\mathbf{Z}^{\mathcal{E}}_\beta\rangle = z_\alpha\delta_{\alpha\beta}$.

Let $\alpha = \gamma_1|\cdots|\gamma_l$ be broken at its ascents (according to $\prec$), then $T^\prec(\alpha)=\gamma_1'|\cdots|\gamma_l'$ and $C^\prec(\alpha)=|\gamma_1'||\cdots||\gamma_l'|$. Thus the change of basis to the Fundamental basis is
\begin{theorem} \label{thm:totcob}
Let $\prec$ be a total order on integers and $P^\prec$ be the powersum basis induced by $\prec$. Then
\[
P_\alpha^\prec= \sum_{I^{\prec}(\alpha)\leq\beta\leq C^\prec(\alpha)}(-1)^{\hgt(\beta,\alpha)}|\sdr(\beta,\alpha)|F_\beta.
\]
\end{theorem}
The proof of this theorem is in Section \ref{sec:NCQtotord}.

For a subposet $\mathcal{R}$, induced by $\prec$, as defined above, we define the reverse subposet of $\mathcal{R}$, denoted as $\overleftarrow{\mathcal{R}}$, with the cover relation $\alpha\lessdot_{\overleftarrow{\prec}}\beta$ if $\beta$ is obtained from $\alpha$ by summing two ascending adjacent parts with respect to $\preceq$, i.e. $\beta = (a_1,\dotsc,a_i+a_{i+1},\dotsc,a_k)$ \textit{only if} $a_i\preceq a_{i+1}$.

\begin{example}
Let $\mathcal{A}=\overleftarrow{\mathcal{D}}$ be the reverse subposet of $\mathcal{D}$. Then define the ascending quasipowersum function as 
\[
P^\mathcal{A}_\alpha = \sum_{\alpha\leq\beta\leq C^\mathcal{A}(\alpha)}C_{\alpha\beta}M_\beta.
\]
\end{example}

Let $\overleftarrow{\alpha}$ denote the reverse composition, then it is easy to see that $\overleftarrow{C^\mathcal{D}(\alpha)}=C^\mathcal{A}(\overleftarrow{\alpha})$.

In \cite{luoto2013introduction}, the star involution ($\star$), omega involution ($\omega$). and $\psi$ involution are defined as
\begin{itemize}
    \item $(M_\alpha)^\star = M_{\overleftarrow{\alpha}}$
    \item $\omega(F_\alpha) = F_{\alpha^t}$
    \item $\psi = \star\circ \omega=\omega\circ\star$
\end{itemize}

\begin{theorem}
Let $\succ$ be a total order on integers and $P^\succ$ be the powersum basis induced by $\succ$. Then,
\begin{enumerate}
    \item $(P^\succ_\alpha)^\star 
 = P^{\overleftarrow{\succ}}_{\overleftarrow{\alpha}}$
    \item $\omega(P_{\alpha}^\succ)=\varepsilon_\alpha P_{\overleftarrow{\alpha}}^\succ$
    \item $\psi(P^\succ_\alpha)=\varepsilon_\alpha P^{\overleftarrow{\succ}}_\alpha$
\end{enumerate}
where $\varepsilon_\alpha=(-1)^{n-\len(\alpha)}$
\end{theorem}

\begin{proof}

(1) By definition 
\begin{align*}
(P^\succ_\alpha)^\star &= \left(\sum_{\alpha\leq\beta\leq C^\succ(\alpha)}C_{\alpha\beta}M_\beta\right)^\star \\
\end{align*}
Note that $C_{\alpha\beta} = C_{\overleftarrow{\alpha}\overleftarrow{\beta}}$ because $m_i(\alpha)$ is the same no matter the order of the parts in $\alpha$, and $c_L(\overleftarrow{\alpha}, \overleftarrow{\beta})=\overleftarrow{c_L(\alpha,\beta)}$, but we are only interested in the product of the factorial of all the parts. Let $\overleftarrow{\succ}=\prec$, then,
\[
(P^\succ_\alpha)^\star 
= \sum_{\overleftarrow{\alpha}\leq\beta\leq C^\prec(\overleftarrow{\alpha})}C_{\overleftarrow{\alpha}\beta}M_\beta = P^\prec_{\overleftarrow{\alpha}}.
\]

(2) Recall that the alternate definition of the omega involution is $\omega(f)=(-1)S(f)$, where $f$ is a quasisymmetric function and $S$ is the antipode. Substituting $f$ with a quasisymmetric powersum function and using Theorem \ref{cor:antipode} yields the above.

(3) This follows from (1) compose (2).
\end{proof}

\begin{remark}
Recall that the quasisymmetric analogue of the forgotten basis of Sym is the Essential Basis $E_\alpha$ and is defined as $E_\alpha = \sum_{\beta\geq\alpha}M_\beta$. Since $\psi(M_\alpha)=-E_\alpha$, the change of basis to the Essential basis is 
\[
P_\alpha^\succ = \sum_{\alpha\leq_{\overleftarrow{\succ}}\beta}\varepsilon_\beta C_{\alpha\beta}E_\beta.
\]
\end{remark}

\section{Preliminaries for NCSym and NCQSym}

\subsection{NCSym}
The Hopf algebra of symmetric functions in non-commuting variables, denoted $\ncsym$, was first defined by Wolf in \cite{wolf1936symmetric}. This space should not be confused with $\nsym$ from Subsection \ref{ssec:Zbasis}.  Function in NCSym are indexed by set partitions. A \textit{set partition} $\phi=\{\mathrm{B}_1,\ldots,\mathrm{B}_k\}$ of $[n]$ is a partition of $[n]$ into nonempty disjoint subsets. Of course being a partition means that order doesn't matter, but we write our set partition (for convenience) as $\mathrm{B}_1/\cdots/\mathrm{B}_k$ where $|\mathrm{B}_1|\geq\cdots\geq |\mathrm{B}_k|$ and $\min(\mathrm{B}_i)>\min(\mathrm{B}_{i+1})$ whenever $|\mathrm{B}_i|=|\mathrm{B}_{i+1}|$. The poset, $\Bar{\mathcal{P}} $, on set partitions, $\phi$ and $\psi$, has the covering relation $\phi\lessdot\psi$ if $\psi = \mathrm{B}_1/\cdots/\mathrm{B}_i\cap\mathrm{B}_j/\cdots/\mathrm{B}_k$.

In \cite{Rosas2006SymmetricVariables}, Rosas and Sagan defined many bases analogous to those in Sym, here we review two. The \textit{symmetric monomial basis in non-commuting variables} is defined as 
\[
m_\phi = \sum_{\substack{i_j=i_{j'} \text{ iff} \\ j,j'\in B_l}}x_{i_1}\cdots x_{i_n}
\]

Then the change of basis formulae of a \textit{symmetric powersum function in noncommuting variables} to a symmetric monomial function in noncommuting variables is 
\begin{equation}\label{eqn:NCpwrsm}
    p_\phi = \sum_{\phi\leq\psi} m_\psi.
\end{equation}

\subsection{NCQSym and Morphisms}
Functions in $\ncqsym$ are indexed by set compositions of $[n]$, or ordered set partitions: a list $\Phi = (B_1,\ldots,B_k)$ of disjoint nonempty sets satisfying $\bigcup_i B_i = [n]$. When there is no danger of confusion, we compress the notation for set compositions as usual, e.g., $(\{5\},\{1,3\},\{2\},\{4\})$ is written as $5|13|2|4$. The refinement poset $\widetilde{\mathcal{P}}$ on set compositions has the cover relation $\Phi\lessdot\Psi$ if $\Psi = (B_1,\ldots,B_i\cup B_{i+1},\ldots, B_j)$. 

Denote by $\rho$ the map from set compositions of $[n]$ to compositions of $n$ that records the cardinality of each block. For example $\rho(5|13|2|4)=(1,2,1,1)$. Let $\alpha = (a_1,\ldots, a_n)$ be a composition, then $\varrho(\alpha)$ is the set composition $\Phi=(B_1,\ldots,B_k)$ where $k$ is the number of distinct parts in $\alpha$ and $i\in B_{|\{ a_l:a_l<a_i \}|+1}$ for every $i$ in $[n]$. For example $\varrho(1,6,4,3,6)=1|4|3|25$.

The Quasisymmetric functions in non-commuting variables $X = (x_1,x_2,\ldots)$, denoted $\ncqsym$, is the space spanned by the formal series
\begin{equation*}
    M_\Phi=\sum_{\alpha:\varrho(\alpha)=\Phi} x_{a_1}x_{a_2}\cdots x_{a_n}.
\end{equation*}
The map $\rho$ from set compositions to compositions induces a Hopf algebra surjection that we will also denote as $\rho$ with $\rho:\ncqsym\to \qsym$ given by $\rho(M_\Phi)=M_{\rho(\Phi)}$.

Analogous to the commutative story, there is a Hopf inclusion of NCSym to NCQSym given by
\begin{equation} \label{eqn:NCtoNCQ}
    m_\phi = \sum_{\Phi:\sort(\Phi)=\phi} M_\Phi.
\end{equation}

\section{Basic Definitions} \label{sec:ncqps}

Given a set composition $\Phi$, a (matrix) filling $\mathsf{F}_\Phi$ places a block $B_i$ in some column along row $j$ for every $j$. Define the set compositions $\row(\mathsf{F}_\Phi)$ and $\col(\mathsf{F}_\Phi)$ by analogy with the integer composition story. For example, $\col (\mathsf{F}_\Phi) =(A_1,\dotsc,A_k)$, then $A_1$ is the union of all blocks $B_j$ of $\Phi$ appearing in the first column of $\mathsf{F}_\Phi$.

We want to define a quasisymmetric powersum function in non-commuting variables based off fillings that compare blocks. To this end, let $\min(B_i)$ be the smallest integer in $B_i$, and define a total order on disjoint sets  $A>_{\widetilde{D}}B$ if and only if either $|A|>|B|$, or $|A|=|B|$ and $\min(A)<\min(B)$.

\begin{definition} \label{def:ldd}
A \emph{Labelled Diagonal Descending (LDD) filling} of $\Phi$ is an assignment of the blocks $B_i$ to entries of a matrix such that
\begin{enumerate}
    \item $B_1$ is in the first column.
    \item $B_i$ is in row $i$. 
    \item $B_{i+1}$ can be in the same column as $B_i$ if $B_i>_{\widetilde{D}}B_{i+1}$. Otherwise, $B_{i+1}$ is in the column to the right of $B_i$.
\end{enumerate}
\end{definition}

Let $\ldd(\Phi)$ be the set of all $\ldd$ fillings with the row reading of $\Phi$. Note that if $\mathsf{F}$ and $\mathsf{F}'$ are two $\ldd$ fillings such that $\col(\mathsf{F})=\col(\mathsf{F}')$ and $\row(\mathsf{F})=\row(\mathsf{F}')$, then $\mathsf{F}=\mathsf{F'}$. In other words every filling is unique to its row and column sum.

\begin{example}
The $\ldd$ fillings with $\row(\mathsf{F}_\Phi)=14|2|3$ are
$$
\begin{array}{|ccc}
14 & \gdot  & \gdot  \\
\gdot  & 2 & \gdot  \\
\gdot  & \gdot  & 3 \\
\hline
14 |& 2| & 3
\end{array}
\qquad
\begin{array}{|ccc}
14 & \gdot  & \gdot  \\
2  & \gdot & \gdot  \\
\gdot  & 3  & \gdot \\
\hline
124 |& 3 & 
\end{array}
\qquad
\begin{array}{|ccc}
14 & \gdot  & \gdot  \\
\gdot  & 2 & \gdot  \\
\gdot  & 3  & \gdot \\
\hline
14 |& 23 & 
\end{array}\qquad
\begin{array}{|ccc}
14 & \gdot  & \gdot  \\
2  & \gdot & \gdot  \\
3  & \gdot  & \gdot \\
\hline
1234 &  & 
\end{array}
$$
Now for a more exciting example, $\ldd(5|13|4|2)$ has the fillings 
$$
\begin{array}{|cccc}
5 & \gdot  & \gdot & \gdot \\
\gdot & 13  & \gdot & \gdot \\
\gdot  & \gdot & 4 & \gdot \\
\gdot  & \gdot  & \gdot & 2 \\
\hline
5 |& 13 |& 4| & 2
\end{array}
\qquad
\begin{array}{|cccc}
5 & \gdot  & \gdot & \gdot \\
\gdot & 13  & \gdot & \gdot \\
\gdot  & 4 & \gdot & \gdot \\
\gdot  & \gdot  & 2 & \gdot \\
\hline
5 |& 134 |& 2| & 
\end{array}
\qquad
$$
\end{example}

\begin{definition} \label{def:WPfill}
Quasisymmetric Powersum Function in NCQSym is defined as
\begin{equation} \label{eqn:Pldd}
P_\Phi = \sum_{\mathsf{F}_\Phi \in \ldd(\Phi)}M_{\col(\mathsf{F}_\Phi)}. 
\end{equation}
\end{definition}

\begin{example}
The previous example yields $P_{14|2|3}=M_{14|2|3}+M_{124|3}+M_{14|23}+M_{1234}$ and $P_{5|13|4|2}=M_{5|13|4|2}+M_{5|134|2}$.
\end{example}

Let $\widetilde{\mathcal{D}}$ be a subposet of $\widetilde{\mathcal{P}}$ on set compositions with the cover relation $(B_1,\ldots,B_i,B_{i+1},\ldots,B_k)\lessdot_{\widetilde{\mathcal{D}}} (B_1,\ldots,B_i\cup B_{i+1},\ldots,B_k)$ only if $B_i>_{\widetilde{D}}B_{i+1}$. Then the non-commutative analogue of Definition \ref{def:pwrsm},

\begin{definition} \label{def:wposet}
Quasisymmetric Powersum Function in NCQSym is defined as
$$P_\Phi = \sum_{\Phi \leq_{\widetilde{\mathcal{D}}} \Psi}M_\Psi= \sum_{\Phi \leq \Psi \leq C(\Phi)}M_\Psi $$
where $C(\Phi)$ is the greatest element in the poset $\widetilde{\mathcal{D}}$ above $\Phi$.
\end{definition}

\begin{example}
Let $\Phi = 1|3|25|6|4$, then $P_{1|3|25|6|4}=M_{1|3|25|6|4}+M_{13|25|6|4}+M_{1|3|256|4}+M_{13|256|4}$
\end{example}

\section{Basic Theorems and Properties}
From Definition \ref{def:wposet} it is clear that $P_\Phi$ is a basis of $\ncqsym$. In the coming sections we'll show that $P_\Phi$ refines the symmetric powersum basis of NCSym and projects to the quasisymmetric powersum basis $P_\alpha$. Finally we'll develop product and coproduct formulas, and generalizations $P^{\rhd}_\Phi$ of $P_\Phi$ given any total order on disjoint sets.

\subsection{Noncommuting Quasipowersum functions refine the Noncommuting Powersums functions} \label{sec:NCQrefine}
Let us expand a symmetric powersum function in NCSym in terms of quasisymmetric monomial function in NCQSym by combining (\ref{eqn:NCpwrsm}) and (\ref{eqn:NCtoNCQ})

\begin{equation} \label{def:NCptoM}
    p_\phi = \sum_{\phi\leq\psi}\sum_{\Psi:\sort(\Psi)=\psi} M_\Psi.
\end{equation}

 We define a \textit{labelled single row filling}, denoted as $\lsr$, as a filling where the row reading is a set partition and every row has only one entry and the column sum is a set composition. Denote $\lsr(\phi)$ as the set of all $\lsr$ fillings with the row reading of $\phi$. $\lsr$ fillings have the property that for any set-partition $\phi$ and set composition $\Psi$ there is at most one filling 
with the row reading of $\phi$ and column sum of $\Psi$.

\begin{example}
$\lsr(13|2)$ has the fillings
$$
\begin{array}{|cc}
13 & \gdot\\
\gdot  & 2\\
\hline
13 |& 2 
\end{array}
\qquad
\begin{array}{|cc}
\gdot & 13\\
2  & \gdot\\
\hline
2 |& 13 
\end{array}
\qquad
\begin{array}{|cc}
13 & \gdot \\
2  & \gdot \\
\hline
123 |&  
\end{array}
$$
\end{example}

Thus we define the change of basis as
\begin{equation} \label{eqn:lsrptoM}
p_\phi = \sum_{\mathsf{F} \in \lsr(\phi)}M_{\col (\mathsf{F})}. 
\end{equation}

This is equivalent to \eqref{def:NCptoM} because the union of multiple blocks (or none) follows from the first summand, and allowing the entry of a block to be anywhere follows from the second summand.

Let $\diag:\lsr(\phi) \to \bigcup_{\Phi} \ldd(\Phi)$ (where $\sort(\Phi) = \phi$) be the map of fillings that sorts the rows of a filling so that the filling is diagonal. Then let $\push:\bigcup_{\Phi} \ldd(\Phi) \to \lsr(\phi)$ (where $\sort(\Phi) = \phi$) be a map of fillings such that $\push$ sorts the rows, so that the row reading is a set partition.
\begin{example}
Below is an example of the map $\diag$ with an $\lsr$ filling, and notice here that both of the fillings are unique and have the same column sum.
\[\footnotesize
\begin{array}{|ccccc}
\gdot & \gdot & 45    & \gdot & \gdot \\
\gdot & \gdot & \gdot & 67    & \gdot \\
1     & \gdot & \gdot & \gdot & \gdot \\
\gdot & \gdot & 2     & \gdot & \gdot \\
\gdot & 3     & \gdot & \gdot & \gdot \\
\hline
1 |   & 3 |   & 245| & 67| &
\end{array}
\quad
\underrightarrow{\diag}
\quad
\begin{array}{|ccccc}
1     & \gdot & \gdot & \gdot & \gdot \\
\gdot & 3     & \gdot & \gdot & \gdot \\
\gdot & \gdot & 45    & \gdot & \gdot \\
\gdot & \gdot & 2     & \gdot & \gdot \\
\gdot & \gdot & \gdot & 67    & \gdot \\
\hline
1 |   & 3 |   & 245| & 67| &
\end{array}
\]
\end{example}

\begin{theorem} \label{thm:toncsym}
$P$ basis refines the symmetric powersum function in NCSym. In other words
\begin{equation} \label{eqn:NCQrefine} 
p_\phi = \sum_{\Phi:\sort(\Phi)=\phi} P_\Phi.
\end{equation}
\end{theorem}
\begin{proof}
By substituting \eqref{eqn:lsrptoM} and \eqref{eqn:Pldd} into \eqref{eqn:NCQrefine}, it is equivalent to prove
\[
\sum_{\widetilde{\mathsf{F}} \in \lsr(\phi)}M_{\col (\widetilde{\mathsf{F}})} =  \sum_{\Phi:\sort(\Phi)=\phi} \sum_{\mathsf{F} \in \ldd(\Phi)}M_{\col(\mathsf{F})}.
\]
We will prove that the above is true by showing that $\push$ and $\diag$ are column-reading-preserving bijections between the fillings of $\lsr(\phi)$ and $\bigcup_{\phi}\ldd(\Phi)$ by showing that $\diag\circ\push=\push\circ\diag=id$.

Fix a filling $\widetilde{\mathsf{F}}\in \lsr(\Phi)$. $\diag$ maps $\widetilde{\mathsf{F}}$ to a $\ldd$ filling such that $\col(\widetilde{\mathsf{F}})=\col(\diag(\widetilde{\mathsf{F}}))$ and the rows may be permuted. $\push$ maps $\diag(\widetilde{\mathsf{F}})$ by placing the row with the greatest entry and permutes it with the top row, then takes the row with the second greatest entry and permutes that with the second row and so on, thus resulting in the row reading of $\push(\diag(\widetilde{\mathsf{F}}))$ being a set partition. Since the rows don't change entries (only permutes) under the maps, $\row(\push(\diag(\widetilde{\mathsf{F}})))=\row(\widetilde{\mathsf{F}})$. Knowing that the column sum is preserved under both maps, i.e. $\col(\push(\diag(\widetilde{\mathsf{F}})))=\col(\widetilde{\mathsf{F}})$, there is at most one unique $\lsr$ filling for a given row reading and column sum. Thus $\push(\diag(\widetilde{\mathsf{F}}))=\widetilde{\mathsf{F}}$. The same sort of argument is used for the other direction.
\end{proof}

\subsection{Product and Coproduct}

Let $B_1,\ldots,B_j,A_1,\ldots,A_k$ be disjoint sets. Then the \textit{shuffle} of lists of disjoint sets  $(B_1,\cdots,B_k)$ and $(A_1,\cdots,A_j)$ is defined recursively by
 \begin{enumerate}
    \item $\emptyset\shuffle (B_1,\cdots,B_j) = (B_1,\cdots,B_j) \shuffle \emptyset = (B_1,\cdots,B_j)$
    \item $(B_1,\cdots,B_j)\shuffle (A_1,\cdots,A_k) = B_1\big|\big[(B_2,\ldots,B_j)\shuffle (A_1,\cdots,A_k)\big] + A_1\big|\big[ (B_1,\cdots,B_j)\shuffle (A_2,\ldots,A_k) \big]$.  
\end{enumerate}
A \textit{quasishuffle} of lists of disjoint sets, denoted $\qshuf$, uses the same initial condition 1, but adds the term $+ (A_1\cup B_1)\big|\big[(B_2,\ldots,B_j) \qshuf (A_2,\ldots,A_k) \big]$ to recursion 2. Let $\Psi^{\uparrow n}$ denote the operation of adding $n$ to every element of every block of $\Psi$, a set composition of $\{n+1,\ldots,n+m\}$. Then the shifted shuffle of two set partitions $\Phi$ and $\Psi$, of $[n]$ and $[m]$ respectively, is $\Phi\shuffle\Psi^{\uparrow n}$, every shuffle is then a set composition of $[n+m]$. Finally let $\Gamma = (C_1,\ldots,C_l)$, if for a subsequence $\Gamma' = (C_{i_1},\ldots,C_{i_k})$ of $\Gamma$, with $m=|C_{i_1}|+\ldots+|C_{i_k}|$, let the standardization of $\Gamma'$, denoted $\Gamma'^{\downarrow }$, be the set composition of $[m]$ where the relative order of the elements across all the $C_{i_j}$ is preserved. 

We extend the notions of shifting and standardizing to fillings, denoted $\mathsf{F}^{\uparrow n}$ and $\mathsf{F}^{\downarrow}$. When the shifting factor $n$ is clear from context, we suppress it. Furthermore, given two $\ldd$ fillings $\mathsf{F}_1$ and $\mathsf{F}_2$, we define a quasishuffle of fillings on $\mathsf{F}_1$ and $\mathsf{F}_2^{\uparrow}$ just like for integer matrix fillings, by taking unions of blocks in place of sums of integers.
\begin{example}
\[
\begin{aligned}\footnotesize
\begin{array}{|cc}
13 & \gdot    \\
\gdot  & 2 \\
\hline
13 |& 2| 
\end{array}
\quad\qshuf\quad
\begin{array}{|c}
 45  \\
\hline
 45|  
\end{array}
 = & \quad
\begin{array}{|c}
13|    \\
\hline
13|  
\end{array}
\left(
\begin{array}{|c}
2    \\
\hline
2|  
\end{array}
\quad \qshuf \quad
\begin{array}{|c}
 45  \\
\hline
 45|  
\end{array}
\right)
\quad  + \quad
\begin{array}{|c}
45    \\
\hline
45|   
\end{array}
\left(
\begin{array}{|cc}
13 & \gdot    \\
\gdot  & 2 \\
\hline
13 |& 2| 
\end{array}
\quad \qshuf \quad
\emptyset\right)\\ & \qquad + 
\quad
\begin{array}{|c}
13  \\
45   \\
\hline
1345 | 
\end{array}
\left(
\begin{array}{|c}
2    \\
\hline
2|  
\end{array}
\quad\qshuf\quad\emptyset
\right)\\
=& \quad
\begin{array}{|ccc}
13 & \gdot  & \gdot   \\
\gdot  & 2 & \gdot  \\
\gdot  & \gdot & 45  \\
\hline
13 |& 2|& 45| 
\end{array}
\quad + \quad
\begin{array}{|ccc}
13 & \gdot  & \gdot   \\
\gdot  & 45 & \gdot  \\
\gdot  & \gdot & 2  \\
\hline
13 |& 45|& 2| 
\end{array}
\quad + \quad
\begin{array}{|ccc}
45 & \gdot  & \gdot   \\
\gdot  & 13 & \gdot  \\
\gdot  & \gdot & 2  \\
\hline
13 |& 45|& 2| 
\end{array}
\\ & \qquad + \quad
\begin{array}{|cc}
13 & \gdot     \\
45  & \gdot   \\
\gdot  & 2   \\
\hline
1345 |& 2| 
\end{array}
\quad + \quad
\begin{array}{|cc}
13 & \gdot     \\
\gdot  & 45   \\
\gdot  & 2   \\
\hline
13 |& 245|  
\end{array}
\end{aligned}
\]
\end{example}

Let $\Phi $ and $\Psi$ be set compositions of $n$ and $m$ respectively. Recall the quasisymmetric monomial basis in non-commuting variables has a shifted quasishuffle product and a standardized deconcatenate coproduct,
\begin{equation*} 
    M_\Phi M_\Psi = \sum_{\Upsilon\in \Phi \qshuf \Psi^{\uparrow n}} M_{\Upsilon},\quad \Delta(M_\Phi)=\sum_{i=0}^k M_{(B_1|\cdots|B_i)^{\downarrow}}\otimes M_{(B_{i+1}|\cdots|B_k)^{\downarrow}}.
\end{equation*}
and the fillings analogue
\begin{equation} \label{eqn:Mprod}
    M_{\col(\mathsf{F}_1)} M_{\col(\mathsf{F}_2)} = \sum_{\mathsf{F}\in \mathsf{F}_1 \qshuf \mathsf{F}_2^{\uparrow}} M_{\col(\mathsf{F})},\quad \Delta(M_{\col(\mathsf{F})})=\sum_{\mathsf{F} = \mathsf{F}_1|\mathsf{F}_2} M_{\col(\mathsf{F}_1^{\downarrow})}\otimes M_{\col(\mathsf{F}_2^{\downarrow})}.
\end{equation}

\begin{theorem} \label{thm:prod}
Let $\Phi$ and $\Psi$ be two set compositions of $n$ and $m$. The powersum basis of $\ncqsym$ has a shifted shuffle product, i.e.
\begin{equation*} \label{eqn:Pprod}
P_\Phi P_\Psi =\sum_{\Upsilon \in \Phi \shuffle \Psi^{\uparrow n} } P_\Upsilon .
\end{equation*}
\end{theorem}

\begin{proof}
From \eqref{eqn:Pldd} and \eqref{eqn:Mprod},
\begin{align*}
    P_\Phi P_\Psi &= \left( \sum_{\mathsf{F}_1\in\ldd(\Phi)} M_{\col(\mathsf{F}_1)} \right)\left( \sum_{\mathsf{F}_2\in\ldd(\Psi)} M_{\col(\mathsf{F}_2)} \right)  \\
    &=\sum_{\substack{\mathsf{F}_1\in\ldd(\Phi),\mathsf{F}_2\in\ldd(\Psi)^{\uparrow}\\\mathsf{F}\in\mathsf{F}_1\qshuf\mathsf{F}_2}}M_{\col(\mathsf{F})}.
\end{align*}
The proof will be carried out the same way as in Section \ref{sec:prodco}, however in this case we need not worry about the coefficient since the coefficient is either 1 or 0. Thus, with abuse of notation, we will show that the sets $\ldd(\Phi)\qshuf\ldd(\Psi)^{\uparrow}$ and $\ldd(\Phi\shuffle\Psi^{\uparrow n})$ are equivalent.

Fix $\mathsf{F}\in\mathsf{F}_1\qshuf\mathsf{F}_2$ where $\mathsf{F}_1\in \ldd(\Phi)$ and $\mathsf{F}_2\in \ldd(\Psi)^{\uparrow}$, then $\row(\mathsf{F})\in\row(\mathsf{F}_1)\shuffle\row(\mathsf{F}_2^{\uparrow})$, which by definition $\row(\mathsf{F}) \in \Phi\shuffle\Psi^{\uparrow n}$. Since the quasishuffle of two $\ldd$ fillings is a $\ldd$ filling, $\mathsf{F}\in\ldd(\Phi\shuffle\Psi^{\uparrow n})$. Conversely fix a filling $\mathsf{F}\in\ldd(\Phi\shuffle\Psi^{\uparrow n})$, then from this filling we construct two fillings: $\mathsf{F}_1$ by removing all entries that has an integer greater than $n$ (then remove all empty columns and rows), and $\mathsf{F}_2$ by removing all entries that has an integer less than or equal to $n$ (then remove all empty columns and rows) and subtracting $n$ from all integers. Thus $\row(\mathsf{F}_1)=\Phi$ and $\mathsf{F}_1$ is an $\ldd$ filling, thus $\mathsf{F}_1\in\ldd(\Phi)$ and likewise for $\mathsf{F}_2$, thus $\mathsf{F}\in \mathsf{F}_1\shuffle\mathsf{F}_2$. Thus the sets are equivalent,
\begin{equation*}
    \sum_{\substack{\mathsf{F}_1\in\ldd(\Phi),\mathsf{F}_2\in\ldd(\Psi)^{\uparrow}\\\mathsf{F}\in\mathsf{F}_1\qshuf\mathsf{F}_2}}M_{\col(\mathsf{F})} = \sum_{\mathsf{F}\in \ldd(\Phi\shuffle\Psi^{\uparrow n})}M_{\col(\mathsf{F})} =\sum_{\Upsilon\in \Phi\shuffle\Psi^{\uparrow n}}P_\Upsilon
\end{equation*}
\end{proof}

\begin{theorem} \label{thm:NCQcoprod}
Let $\Phi=B_1|\cdots|B_k$ be a set composition of $[n]$. Then the quasisymmetric powersum basis in non-commuting variables has a standardized deconcatenation coproduct,
\begin{equation*}
    \Delta(P_\Phi)=\sum_{i=0}^k P_{(B_1|\cdots|B_i)^{\downarrow}}\otimes P_{(B_{i+1}|\cdots|B_k)^{\downarrow}}.
\end{equation*}
\end{theorem}
\begin{proof}
Definition \ref{def:wposet} and \eqref{eqn:Mprod} yields
\begin{equation*}
    \Delta(P_\Phi)=\Delta\left(\sum_{\mathsf{F}\in\ldd(\Phi)}M_{\col(\mathsf{F})}\right)=\sum_{\mathsf{F}\in\ldd(\Phi)}\sum_{\mathsf{F}=\mathsf{F}_1|\mathsf{F}_2}M_{\col(\mathsf{F}_1^{\downarrow})}\otimes M_{\col(\mathsf{F}_2^{\downarrow})}. 
\end{equation*}
Let $m$ be the degree of $\col(\mathsf{F}_1)$ and fix $m$. Let $\mathsf{F}_1'$, such that $\row(\mathsf{F}_1')=\col(\mathsf{F}_1')=(B_1,\cdots,B_j)$ (which means $\mathsf{F}_1'$ is a "diagonal" filling), then every filling of degree $m$ comes from moving entry $B_{i+1}$ to the column of $B_i$ when $B_i>_{\widetilde{\mathcal{D}}}B_{i+1}$. Thus the set of fillings is $\ldd(B_1,\ldots,B_j)$ and likewise for the right side, thus
\[
\sum_{\mathsf{F}\in\ldd(\Phi)}\sum_{\mathsf{F}=\mathsf{F}_1|\mathsf{F}_2}M_{\col(\mathsf{F}_1^{\downarrow})}\otimes M_{\col(\mathsf{F}_2^{\downarrow})} =  \sum_{i=0}^k\sum_{\substack{\mathsf{F}_1\in\ldd((B_1|\cdots|B_i)^{\downarrow})\\\mathsf{F}_2\in\ldd((B_{i+1}|\cdots|B_k)^{\downarrow})}} M_{\col(\mathsf{F}_1)}\otimes M_{\col(\mathsf{F}_2)}.
\]
\end{proof}

\begin{remark}
The method of proving the theorems above could be solved more straight forward and perhaps easily if we used the subposet definition instead of the fillings definition. However we choose this approach to demonstrate that one can make a Hopf algebra out of the fillings using this approach and that there may not always be a subposet definition for different Hopf algebas.
\end{remark}

\section{The Projection of the Quasisymmetric Powersum Basis} \label{sec:proj}
Since there is a powersum basis in Sym, NCSym, QSym, and NCQSym, it is natural to explore the projection of $P_\Phi$ onto $\qsym$. As we'll see in Theorem \ref{thm:toqsym}, unlike the story in the monomial basis, it's not the case that $\rho(P_\Phi)=P_{\rho(\Phi)}$.

We first look to when two quasipowersum function in $\ncqsym$ map to the same function in $\qsym$ under $\rho$.

\begin{proposition}
Let $\Phi$ and $\Psi$ be two set compositions composed of blocks $B_1|\cdots|B_k$ and $A_1|\cdots|A_k$ respectively. If the condition that $B_i >_{\widetilde{D}} B_{i+1}$ if and only if $A_i >_{\widetilde{D}} A_{i+1}$ for all $i<k$ holds and $|B_i|=|A_i|$, then $\rho(P_\Phi)=\rho(P_\Psi)$.
\end{proposition}

\begin{proof}
Recall that $M_\Phi$ projects to $\qsym$ by,
\begin{equation}
    \rho(M_\Phi)=M_{\rho(\Phi)}.
\end{equation}
It follows that 
\begin{equation}
    \rho(P_\Phi)=\rho\left(\sum_{\Phi\leq_{\widetilde{\mathcal{D}}}\hat{\Phi}}M_{\hat{\Phi}}\right)=\sum_{\Phi\leq_{\widetilde{\mathcal{D}}}\hat{\Phi}}M_{\rho\left(\hat{\Phi}\right)}\quad\text{ and } \quad \rho(P_\Psi)=\sum_{\Psi\leq_{\widetilde{\mathcal{D}}}\hat{\Psi}}M_{\rho\left(\hat{\Psi}\right)}.
\end{equation}
For every $\hat{\Phi}$ there exists one and only one $\hat{\Psi}$ such that $\rho(\hat{\Phi})=\rho(\hat{\Psi})$. Fix a $\hat{\Phi}$, then $\Phi\leq\hat{\Phi}$ by combining certain blocks $B_j$ and $B_{j+1}$, since $B_i>_{\widetilde{D}} B_{i+1}$ if and only if $A_i>_{\widetilde{D}} A_{i+1}$ for all $i<k$ holds, then there exists a $\Psi\leq\hat{\Psi}$ by combining blocks $A_j$ and $A_{j+1}$. $\rho(\hat{\Phi})=\rho(\hat{\Psi})$ holds due to the fact that $|B_i|=|A_i|$ for all $i\leq k$ and the same blocks are being combined.
\end{proof}

Let $\Phi = (B_1,\cdots,B_l)$ be a set composition of $[n]$ and let $\mathfrak{B}_k$ be the set of all $B_i$ of size $k$. Then an element $\sigma \in \mathfrak{S}_{\mathfrak{B}_1}\times \mathfrak{S}_{{\mathfrak{B}_2}}\cdots\times \mathfrak{S}_{{\mathfrak{B}_n}}=\mathfrak{S}_\Phi$ acts on $\Phi$ by place permutation. For example, $(35)(24)(4|25|7|13|6)=4|13|6|25|7$. Note that $\rho(\Phi) = \rho(\sigma(\Phi))$. This leads to the following theorem.

\begin{theorem} \label{thm:toqsym} 
For any composition $\alpha$ and set composition $\Phi$ such that $\rho(\Phi)=\alpha$,
\begin{equation} \label{eq:P to P}
P_\alpha = \sum_{\sigma\in\mathfrak{S}_\Phi}\rho(P_{\sigma(\Phi)}).
\end{equation}
\end{theorem}

Let $\Phi=(B_1,\ldots,B_l)$ be a set composition of $[n]$. We call $\Phi$ \textit{strict} if for all $i<j$ with $|B_i|=|B_j|$, we have $B_i>_{\widetilde{\mathcal{D}}} B_j$. For example $\Phi = 2|13|4|5$ is a strict set composition of $[n]$.

\begin{definition}
Let a Strict Labelled Diagonal filling, denoted $\sld$, be an $\ldd$ filling such that the row reading is a strict set composition.
\end{definition}

It is important to note that $\sld$ fillings follow the rule that $B_i$ and $B_{i+1}$ can be in the same column if $|B_i|\geq|B_{i+1}|$. 

\begin{example}
The first two are $\sld$ fillings, but the thirs is only a $\ldd$ filling.
$$
\begin{array}{|cccc}
2 & \gdot  & \gdot & \gdot \\
\gdot & 13  & \gdot & \gdot \\
\gdot  & 4 & \gdot & \gdot \\
\gdot  & \gdot  & 5 & \gdot \\
\hline
2 |& 134 |& 5| & 
\end{array}
\qquad
\begin{array}{|cc}
2 & \gdot   \\
\gdot & 13   \\
\gdot  & 4   \\
\gdot  & 5   \\
\hline
2 |& 1345| 
\end{array}
\qquad
\begin{array}{|cc}
4 & \gdot   \\
\gdot & 13   \\
\gdot  & 2   \\
\gdot  & 5   \\
\hline
4 |& 1235| 
\end{array}.
$$
\end{example}

Note, for a fixed $\Bar{\mathsf{F}}\in\ldd$, some $\sigma\in\mathfrak{S}_\Phi$ preserve the diagonal descent property: $\sigma\Bar{\mathsf{F}}\in\ldd$, but others do not. Next we will show how to generate $\ldd$ fillings from permutations and $\sld$ fillings. 

\begin{example}
Let $\Bar{\mathsf{F}}$ be a filling of $\sld(2|13|4|5)$ on the left. Then the permutations of $\mathfrak{S}_{\Phi}$ that results in an $\ldd$ filling are $\{id, (13)(4)(2),(134)(2))\}$ which are shown below from left to right:
\[
\begin{array}{|cccc}
2     & \gdot & \gdot & \gdot \\
\gdot & 13    & \gdot & \gdot \\
\gdot & 4     & \gdot & \gdot \\
\gdot & 5     & \gdot & \gdot \\
\hline
2 |& 1345 |&  & 
\end{array}
\qquad
\begin{array}{|cccc}
4     & \gdot & \gdot & \gdot \\
\gdot & 13    & \gdot & \gdot \\
\gdot & 2     & \gdot & \gdot \\
\gdot & 5     & \gdot & \gdot \\
\hline
4 |& 1235 |&  & 
\end{array}
\qquad
\begin{array}{|cccc}
5     & \gdot & \gdot & \gdot \\
\gdot & 13    & \gdot & \gdot \\
\gdot & 2     & \gdot & \gdot \\
\gdot & 4     & \gdot & \gdot \\
\hline
2 |& 1345 |&  & 
\end{array}.
\]
The permutation $(14)(3)(2)$ is an example of a permutation that does not result in a $\ldd$ filling. 
\[
\begin{array}{|cccc}
5     & \gdot & \gdot & \gdot \\
\gdot & 13    & \gdot & \gdot \\
\gdot & 4     & \gdot & \gdot \\
\gdot & 2     & \gdot & \gdot \\
\hline
5 |& 1234 |&  & 
\end{array}.
\]
\end{example}

The group $\mathfrak{S}_\Phi$ acts analogously on fillings, permuting matrix entries. For a fixed filling $\Bar{\mathsf{F}}$, we are interested in a subset $\mathfrak{S}_{\Bar{\mathsf{F}}}$ of $\mathfrak{S}_\Phi$ described as follows. Say $\sigma\in\mathfrak{S}_{\Bar{\mathsf{F}}}$ if for every $\sigma B_i$, $\sigma B_j$, and $\sigma B_k$ in the same column, if $|B_i| = |B_j|=|B_k|$ and $\sigma B_k$ above $\sigma B_i$ above $\sigma B_j$, then $B_k>_{\widetilde{\mathcal{D}}}B_i>_{\widetilde{\mathcal{D}}}B_j$.

Let $\mathfrak{B}_{k,j}$ be the set of all $B_i$ of size $k$ in column $j$. The subgroup $\mathfrak{H}_\mathsf{F}\subseteq \mathfrak{S}_\Phi$ is $\mathfrak{S}_{\mathfrak{B}_{1,1}}\times\cdots\times\mathfrak{S}_{\mathfrak{B}_{1,j}}\times\cdots\times\mathfrak{S}_{\mathfrak{B}_{2,1}}\times\cdots\times\mathfrak{S}_{\mathfrak{B}_{n,n}}$. In other words, $\mathfrak{H}_\mathsf{F}$ is the group of all $\sigma$ that preserves the columns and placements of block sizes. Thus if $\mathsf{F}$ is an $\ldd$ filling, then for $\sigma\in\mathfrak{H}_\mathsf{F}$, $\sigma\mathsf{F}$ is not an $\ldd$ filling unless $\sigma = id$. It is easy to see that $|\mathfrak{H}_\mathsf{F}| = \prod_i c_i(\row(\mathsf{F}),\col(\mathsf{F}))$.

Note that the set $\mathfrak{S}_\mathsf{F}\mathsf{F}$ contains a unique $\sld$ filling, $\Bar{\mathsf{F}}$, for any $\ldd$ filling $\mathsf{F}$. Let $\Bar{\mathsf{F}}$ be a $\sld$ filling with row reading $\Phi$. Consider the coset $\mathfrak{S}_\Phi/\mathfrak{H}_\mathsf{F}$, all permutations must involve permutations across columns such that the order in columns is preserved, which is precisely $\mathfrak{S}_\mathsf{F}$. Thus, 
\begin{equation} \label{eqn:setsize}
    |\mathfrak{S}_{\Bar{\mathsf{F}}}|=\frac{|\mathfrak{S}_\Phi|}{|\mathfrak{H}_\mathsf{F}|} = 
    \prod_i \frac{m_i(\alpha)!}{c_i(\row(\Bar{\mathsf{F}}),\col(\Bar{\mathsf{F}})))!}.
\end{equation}

\begin{proof}[Proof of Theorem \ref{thm:toqsym}]
Note that for the set $\mathfrak{S}_\Phi\Phi$ of all set compositions resulting from the action of $\mathfrak{S}_\Phi$ must contain a strict set composition. Since the action is transitive, we let $\Phi$ be a strict set composition for the rest of this proof. We take the right hand side of \eqref{eq:P to P} and substitute Definition \ref{def:WPfill} to get
\begin{equation*}
\rho\left(\sum_{\sigma\in\mathfrak{S}_\Phi}P_{\sigma(\Phi)}\right) 
= \rho\left(\sum_{\sigma\in\mathfrak{S}_\Phi}\sum_{\mathsf{F}\in\ldd(\sigma(\Phi))}M_{\col(\mathsf{F})}\right)
\end{equation*}
Fix a permutation $\sigma\in\mathfrak{S}_\Phi$ and a filling $\mathsf{F}\in\ldd(\sigma(\Phi))$. Then we can permute the entries of $\mathsf{F}$ by $\sigma^{-1}$, this results in a $\sld$ filling $\Bar{\mathsf{F}}$. This means that $\sigma\in\mathfrak{S}_{\Bar{\mathsf{F}}}$ because $\mathfrak{S}_{\Bar{\mathsf{F}}}$ are all the permutations that permute the entries of a $\sld$ filling to result in a $\ldd$ filling, and we know that $\sigma^{-1}(\mathsf{F})=\Bar{\mathsf{F}}$ which implies that $\mathsf{F}=\sigma(\Bar{\mathsf{F}})$. Thus, 
\begin{equation*}
    \rho\left(\sum_{\sigma\in\mathfrak{S}_\Phi}\sum_{\mathsf{F}\in\ldd(\sigma(\Phi))}M_{\col(\mathsf{F})}\right) =\rho\left(\sum_{\Bar{\mathsf{F}}\in\sld(\Phi)}\sum_{\Bar{\sigma}\in\mathfrak{S}_{\Bar{\mathsf{F}}}}M_{\col(\Bar{\sigma}(\Bar{\mathsf{F}}))}\right).
\end{equation*}
Recall that $\rho(\Phi)=\rho(\sigma(\Phi))$, thus
\begin{equation*}
    \rho\left(\sum_{\Bar{\mathsf{F}}\in\sld(\Phi)}\sum_{\Bar{\sigma}\in\mathfrak{S}_{\Bar{\mathsf{F}}}}M_{\col(\Bar{\sigma}(\Bar{\mathsf{F}}))}\right)
    =\sum_{\Bar{\mathsf{F}}\in\sld(\Phi)}|\mathfrak{S}_{\Bar{\mathsf{F}}}|M_{\rho(\col(\Bar{\mathsf{F}}))}.
\end{equation*}
Consider when $\rho$ acts on a filling $\Bar{\mathsf{F}}$ by replacing the entry $B_i$ with $|B_i|$. $\rho$ is a bijection between $\sld(\Phi)$ and $\sd(\alpha)$ because $\rho(\Phi)=\alpha$ and the rule ``$B_i$ is in the same column as $B_{i+1}$ when $|B_i|\geq|B_{i+1}|$" maps to the rule ``$a_i$ is in the same column as $a_{i+1}$ when $a_i\geq a_{i+1} $". Finally, since $\mathfrak{S}_{\Bar{\mathsf{F}}}$ depends only on the sizes of the entries of $\Bar{\mathsf{F}}$, $|\mathfrak{S}_{\Bar{\mathsf{F}}}|=|\mathfrak{S}_{\widetilde{\mathsf{F}}}|$ such that $\rho(\Bar{\mathsf{F}})=\widetilde{\mathsf{F}}$. Thus we get the following,
\begin{equation*}
    \sum_{\Bar{\mathsf{F}}\in\sld(\Phi)}|\mathfrak{S}_{\Bar{\mathsf{F}}}|M_{\rho(\col(\Bar{\mathsf{F}}))} = \sum_{\widetilde{\mathsf{F}}\in \sd(\alpha)}|\mathfrak{S}_{\widetilde{\mathsf{F}}}|M_{\col(\widetilde{\mathsf{F}})}= P_\alpha
\end{equation*}
\end{proof}

Now we determine the image of $P_\Phi$ in QSym in the quasisymmetric fundamental basis. This is the last theorem needed to prove Theorem \ref{thm:MNrule}. We start by defining two operations on set compositions. Let $\Phi = (B_1,\ldots,B_k)$ be a set composition of $[n]$ with $\rho(\Phi)=(a_1,\ldots,a_k)$. Break $\Phi$ at nondescents, $B_i<_{\widetilde{\mathcal{D}}}B_{i+1}$ to build the sequence $\Bar{\Phi} = (\Phi_1,\ldots,\Phi_l)$ and $\rho(\Bar{\Phi}) = (\alpha_1,\ldots,\alpha_l)$. (Note $\alpha_i=(a_{j_1},\ldots,a_{j_1+j_2})$ is a subsequence of $\alpha = \rho(\Phi)$ satisfying $B_{j_1-1}<_{\widetilde{\mathcal{D}}}B_{j_1}>_{\widetilde{\mathcal{D}}}B_{j_1+1}>_{\widetilde{\mathcal{D}}}\cdots>_{\widetilde{\mathcal{D}}} B_{j_1+j_2}<_{\widetilde{\mathcal{D}}}B_{j_1+j_2+1}$.) Associate to $\rho$ the functions $\rho_C$ and $\rho_T$ defined by
\begin{equation*} 
    \rho_C(\Phi)=(|\alpha_1|,|\alpha_2|,\cdots,|\alpha_l|) \quad \text{and} \quad \rho_T(\Phi)= \alpha_1'|\alpha_2'|\cdots|\alpha_l'
\end{equation*}
where $\alpha_i'$ denotes ribbon transpose. These yield minimal and maximal compositions associated to $\alpha$ in a way that will become clear shortly. For example, $\rho_C(2|5|14|36|7)=(2,5)$ and $\rho_T(2|5|14|36|7)=(2,1,2,2)$. 

\begin{remark} \label{rm:subsetTC}
In subset notation, $\rho(\Phi)=\rho_C(\Phi)\cup D$, and $\rho_T(\Phi)=\rho_C(\Phi)\cup D^\complement$, where $D$ represents the ``descending divisions".
\end{remark}

\begin{theorem} \label{thm:pwrsmim}
Let $\Phi$ be a set composition and $\mu_{\mathcal{P}}$ be the M\"obius function on the refinement poset $\mathcal{P}$ on compositions, then
\begin{equation} \label{eqn:proF}
\rho(P_\Phi)=\sum_{\rho_T(\Phi)\leq\beta\leq\rho_C(\Phi)}\mu_{\mathcal{P}}(\beta,\rho_C(\Phi))F_\beta.
\end{equation}
\end{theorem}

\begin{proof}
One can see that 
\[
\rho(P_\Phi)=\rho\left(\sum_{\Phi\leq\Psi\leq C(\Phi)}M_\Psi\right)=\sum_{\rho(\Phi)\leq\gamma\leq\rho_C(\Phi)}M_\gamma
\]
Using the change of basis from the monomial basis to the fundamental basis we get
\[
\rho(P_\Phi)=\sum_{\rho(\Phi)\leq\gamma\leq\rho_C(\Phi)}\sum_{\gamma\geq\beta} \mu_{\mathcal{P}}(\beta,\gamma)F_\beta.
\]

We gather all the $F_\beta$ terms and write

\[
\rho(P_\Phi)=\sum_{\beta\leq\rho_C(\Phi)}\left( \sum_{\gamma\in G(\Phi,\beta)} \mu_{\mathcal{P}}(\beta,\gamma) \right)F_\beta
\]
where $G(\Phi,\beta)$ is the interval $[\rho(\Phi)\wedge \beta , \rho_C(\Phi)]$. We note that this a boolean lattice so the cardinality of this set is either 1 or an even number. Let $\widetilde{\gamma}=\min(G(\Phi,\beta))$, where $\widetilde{\gamma}\neq\rho_C(\Phi)$. It takes $j$ coarsenings to go from $\widetilde{\gamma}$ to $\rho_C(\Phi)$, thus we can relate all compositions in $G(\Phi,\beta)$ relative to $\widetilde{\gamma}$ by the tuple $(g_1,\ldots,g_j)$ where $g_i=1$ if the two entries are coarsened and $g_i=0$ if the two entries are not coarsened. This means that there are $2^j$ compositions in $G(\Phi,\beta)$ and there is natural pairing of two tuples $(0,g_2,\ldots,g_j)$ and $(1,g_2,\ldots,g_j)$. The resulting compositions have a length difference of 1, which means that $\mu_{\mathcal{P}}(\beta,\gamma_1)+\mu_{\mathcal{P}}(\beta,\gamma_2)=0$ where $\gamma_1$ and $\gamma_2$ are the compositions that come from $(0,g_2,\ldots,g_j)$ and $(1,g_2,\ldots,g_j)$ respectively. Since the cardinality of $G(\Phi,\beta)$ is even and the sum of every pair is 0, the coefficient of $F_\alpha$ is 0.

Now consider the case $\rho(\Phi)\wedge \beta =\rho_C(\Phi)$, then the coefficient is $\mu_{\mathcal{P}}((\beta,\rho_C(\Phi))$. To find the interval for all such $\beta$'s: since $\rho(\Phi)\wedge \beta =\rho_C(\Phi)$, in subset notation $\rho(\Phi)\cap \beta =\rho_C(\Phi)$. This means $\beta \supseteq \rho_C(\Phi)$ and, for every $A \in \rho(\Phi) - \rho_C(\Phi)$, then $A \notin \beta$. Since we want the smallest $\beta$ in the interval, i.e. the largest subset, we should take $\beta = \rho_C(\Phi) \cup (\rho(\Phi) - \rho_C(\Phi))^\complement = \rho_T(\Phi)$.
\end{proof}

\subsection{Proof of Theorem \ref{thm:MNrule} and \ref{thm:totcob}} \label{sec:proof}

Notice that we can express a quasipowersum function in the fundamental basis by substituting \eqref{eqn:proF} into \eqref{eq:P to P}

\begin{equation}\label{eqn:sub1}
P_\alpha = \sum_{\sigma\in\mathfrak{S}_\Phi}\rho(P_{\sigma(\Phi)})=\sum_{\sigma\in\mathfrak{S}_\Phi}\sum_{\rho_T(\sigma(\Phi))\leq\beta\leq\rho_C(\sigma(\Phi))}\mu_{\mathcal{P}}(\beta,\rho_C(\sigma(\Phi)))F_\beta.
\end{equation}

This expression will be needed in the proof of Theorem \ref{thm:MNrule}, and fortunately it turns out that the intervals $[\rho_T(\sigma(\Phi)),\rho_C(\sigma(\Phi))]$ are disjoint for different $\sigma$, which means that it is quite easy to find the coefficient of $F_\beta$.

\begin{lemma} \label{lem:djsets}
Let $\Phi$ be a set composition and $\sigma\in\mathfrak{S}_\Phi$.
Fix $\beta$ such that $\rho_T(\sigma(\Phi))\leq\beta\leq\rho_C(\sigma(\Phi))$. If there exists a $\sigma'\in \mathfrak{S}_\Phi$ such that $\rho_C(\sigma(\Phi))\neq\rho_C(\sigma'(\Phi))$, then $\beta$ is not in the interval $[\rho_T(\sigma'(\Phi)),\rho_C(\sigma'(\Phi))]$.
\end{lemma}

\begin{proof}
We'll do a proof by contradiction and use set notation. If $\rho_C(\sigma(\Phi))\neq\rho_C(\sigma'(\Phi))$ then without loss of generality there exists $A\in\rho_C(\sigma'(\Phi))$, $A\not\in\rho_C(\sigma(\Phi))$. Thus if $\beta$ is less than $\rho_C(\sigma(\Phi))$ and $\rho_C(\sigma'(\Phi))$ then $\beta \supseteq \rho_C(\sigma'(\Phi)) \ni A$. Since $A\in\rho_C(\sigma'(\Phi))$, we must have $A\in \rho(\sigma'(\Phi))=\rho(\Phi)$, but $A\not\in\rho_C(\sigma(\Phi))$ so $A$ is a ``descending division" so $\rho_T(\sigma(\Phi)) \not \ni A$. However if $\beta\geq\rho_T(\sigma(\Phi))$ then $\beta \subseteq \rho_T(\sigma(\Phi))$ so $\beta \not \ni A$, thus a contradiction.
\end{proof}

\begin{proof}[Proof of Theorem \ref{thm:MNrule}]
Note that the largest (with respect to $\mathcal{P}$) composition in the interval of \eqref{eqn:sub1} is $C(\alpha)$ due to the fact that there is a $\Psi\in\mathfrak{S}_\Phi(\Phi)$ such that $\Psi$ is a strict composition thus $\rho_C(\Psi)=C(\alpha)$. Equivalently, the smallest composition in the interval comes from the opposite of a strict set composition which is $T(\alpha)$.

Throughout, we use subset notation for compositions and appeal to Remark \ref{rm:subsetTC}.  Lemma \ref{lem:djsets} implies that all intervals are disjoint, which means that in \eqref{eqn:sub1} the coefficient of $F_\beta$ is (up to sign) the number of $\sigma\in\mathfrak{S}_\Phi$ that results in the same $\rho_C(\sigma(\Phi))$, we denote this set as $\Sigma$. Fix $\beta$ and  $\bar{\sigma}$ such that $\beta\leq\rho_C(\bar{\sigma}(\Phi))$ and let $\Sigma = \{ \sigma \in \mathfrak{S}_\Phi : \bar{\sigma}(B_i)>\bar{\sigma}(B_{i+1})\iff\sigma(B_i)>\sigma(B_{i+1})\}$. Notice that according to $\widetilde{\mathcal{D}}$, if $B_i$ and $B_{i+1}$ are two different sizes and $B_i>B_{i+1}$, then $\sigma(B_i)>\sigma(B_{i+1})$, due to the fact that $\widetilde{\mathcal{D}}$ compares sizes. We observe that, in Algorithm 1, the conditions for whether lines 3, 5, or 7 happens at the $i$th iteration are equivalent to three conditions on $\bar{\sigma}(B_{i-1})$ and $\bar{\sigma}(B_i)$. For the first case, $|\bar{\sigma}(B_{i-1})|\neq|\bar{\sigma}(B_i)|$ if and only if $a_{i-1}\neq a_i$. In the second case, if $\bar{\sigma}(B_{i-1})<_\mathcal{D}\bar{\sigma}(B_i)$ where $|\bar{\sigma}(B_{i-1})|=|\bar{\sigma}(B_i)|$, then obviously $a_{i-1}=a_i$, but this implies that the composition $\rho_C(\bar{\sigma}(\Phi))$ contains $ A$ (written in subset notation). This implies $A\in\beta$, since $\rho_C(\bar{\sigma}(\Phi))\geq\beta$. Thus the $A$\textsuperscript{th} box is above the $(A+1)$\textsuperscript{st} box in the original ribbon of $\beta$ and since $R$ begins at the $(A+1)$\textsuperscript{st} box, there exists a $b_j$ such that $(b_j,\ldots,b_l)$ forms the ribbon $R$. For the last case, let $\bar{\sigma}(B_{i-1})>_\mathcal{D}\bar{\sigma}(B_i)$ where $|\bar{\sigma}(B_{i-1})|=|\bar{\sigma}(B_i)|$, then $A\notin\rho_T(\bar{\sigma}(\Phi))$ which means that $A\notin\beta$ because $\rho_T(\bar{\sigma}(\Phi))\leq \beta$. This implies that the $A$\textsuperscript{th} and $(A+1)$\textsuperscript{st} box are in the same row of the ribbon of $\beta$. Since $R$ starts at the $(A+1)$\textsuperscript{st} box, there does not exist a $b_j$ such that $(b_j,\ldots,b_l)$ forms the ribbon $R$.

Recall $\mathfrak{B}_{k}$, the set of all blocks $B$ of $\Phi$ such that $|B|=k$, then there is a natural bijection, $O$ from $\mathfrak{B}_k$ to the interval $[|\mathfrak{B}_k|]$, satisfying $B>_{\tilde{\mathcal{D}}}B'$ if and only if $O(B)>O(B')$. Thus $\Sigma$ is also enumerated by the number of standard fillings of $D(\beta,\rho(\bar{\sigma}(\Phi)))$ where the entries of a filling, when read from left to right, are increasing (due to the alternate condition of RTa2) and, when read from top to bottom, are decreasing (due to the alternate condition of RTa3). Thus $|\Sigma|=\sdr(\beta,\alpha)$.

Finally the M\"obius function has a $(-1)^{\len(\beta)-\len(\rho_C(\sigma(\Phi))}$ factor and we will show by induction that $(-1)^{\hgt(\beta,\rho(\sigma(\Phi)))}=(-1)^{\len(\beta)-\len(\rho_C(\sigma(\Phi))}$ when $\rho_T(\sigma(\Phi))\leq\beta\leq\rho_C(\sigma(\Phi))$. Consider the base case when $\beta = \rho_C(\sigma(\Phi))$. Then $\rho(\sigma(\Phi))$ refines $\rho_C(\sigma(\Phi))$ so $\hgt(\rho_C(\sigma(\Phi)),\rho(\sigma(\Phi)))=0$ and obviously $\len(\rho_C(\sigma(\Phi))-\len(\rho_C(\sigma(\Phi))=0$. Now suppose that $\hgt(\beta,\rho(\sigma(\Phi)))= \len(\beta)-\len(\rho_C(\sigma(\Phi))) = t$. Let $\len(\gamma) = \len(\beta)+1$ where $\rho_T(\sigma(\Phi))\leq\gamma<\beta\leq\rho_C(\sigma(\Phi))$, then there exists an element $B$ such that $B\notin\beta$ and $B \in \gamma$, since $\rho_T(\sigma(\Phi))\leq\gamma$, $B\notin D$. Let $A$ be maximal such that $A<B$ and $a_1+\cdots+a_k=A$. Then suppose that in the $k$th iteration of RT for $\beta$ the height equals $s$, then in the $k$th iteration of RT for $\gamma$, the height equals $ s+1 $ because the $i$th box is the end of a row of the ribbon $\gamma$. Thus at the final iteration we will get $\hgt(\gamma,\rho(\sigma(\Phi)))= t+1$. Obviously $\len(\gamma)-\len(\rho_C(\sigma(\Phi))) = t+1$, thus ${\hgt(\gamma,\rho(\sigma(\Phi)))}={\len(\gamma)-\len(\rho_C(\sigma(\Phi))}$.
\end{proof}

\section{Generalizations of $P_\Phi$ Using Total Orders} \label{sec:NCQtotord}
With careful work we can generalize a $\ldd$ filling by replacing $>_{\tilde{\mathcal{D}}}$ by a different total ordering $\rhd$ on disjoint integer sets. The resulting powersum $P^{\rhd}$ refines the symmetric powersums in non-commuting variables. Only certain total orders $\rhd$ induce $P^{\rhd}$ that have a shifted shuffle product and a standardized deconcatenate coproduct. 

\begin{definition} \label{def:invtotord}
A total order $\rhd$ on disjoint sets $A$ and $B$ is \textit{shift-invariant} if it has the property that if $A\rhd B$, then $A^{\uparrow k}\rhd B^{\uparrow k}$ for all $k$. A total order $\rhd$ is \textit{standard-invariant} if it has the property that if $A\rhd B $, then $A^\downarrow \rhd B^\downarrow$.
\end{definition}

Studying carefully Definition \ref{def:wposet} and the proofs of Theorem \ref{thm:NCQcoprod} and \ref{thm:prod}, we see that the only crucial properties of $>_{\widetilde{\mathcal{D}}}$ used are its shift and standard invariance. We conclude:

\begin{proposition}
If the total order $\rhd$ is shift-invariant, then $P^\rhd$ has a shifted shuffle product. Likewise, if $\rhd$ is standard-invariant, then $P^\rhd$ has a standardized deconcatenation coproduct.
\end{proposition}

\begin{example}
Let $>_{Med}$ be the total order on disjoint sets $A$ and $B$ by the covering relation
\[
A>_{Med}B=\begin{cases}|A|>|B|& |A|\neq|B|\\ \med(A)>\med(B) & |A|\neq|B| \end{cases}
\]
where $\med(A)$ is the median (rounded up) of $A$. It follows that $>_{\med}$ is a shift-invariant and standard-invariant total order, which means that $P^\med$ has a shifted shuffle product and a standardized deconcatenate coproduct. An interesting question is whether a standard invariant total order is also shift invariant (and vice versa), we leave this question to the reader.
\end{example}

Let $B$ and $A$ be two disjoint sets of different sizes. We say that a total ordering \textit{projects under $\rho$} if there exists a total ordering on integers $\succ$, such that $B\rhd A$ implies $|B|\succ|A|$.

\begin{example}
Define the total order on sets $>_{\widetilde{M}}$ as $B>A$ if $\min(B)>\min(A)$. Then $>_{\widetilde{M}}$ does not project under $\rho$ because $\{3,4\}>_{\widetilde{M}}\{2\}$ and $\{2,4\}<_{\widetilde{M}}\{3\}$. This means that $2>_{M}1$ and $2<_M1$, however this implies that $<_{M}$ is not a total order. We note that $>_{\widetilde{M}}$ is a standard invariant and shift invariant total order.
\end{example}

Let $\Phi = (B_1,\ldots,B_k)$ be a set composition of $n$ with $\rho(\Phi)=(b_1,\ldots,b_k)$. Break $\rho(\Phi)$ when $B_i\rhd B_{i+1}$ so that $\rho(\Phi)=\beta_1|\beta_2|\cdots|\beta_l$ where $\beta_i=(b_{j_1},\ldots,b_{j_1+j_2})$ with  $B_{j_1-1}\lhd B_{j_1}\rhd B_{j_1+1}\rhd\cdots\rhd B_{j_1+j_2}\lhd B_{j_1+j_2+1}  $. Then we define $\rho_C^\rhd(\Phi)=(|\beta_1|,|\beta_2|,\cdots,|\beta_l|)$ and $\rho_T^\rhd(\Phi)= \beta_1'|\beta_2'|\cdots|\beta_l'$, where $\beta_i'$ is the transpose of $\beta_i$.

\begin{theorem}\label{thm:totpro}
Let $\rhd$ be a total order that projects to $\prec$. Then
\begin{itemize}
    \item $P_\alpha^\prec = \sum_{\sigma\in\mathfrak{S}_\Phi}\rho(P_{\sigma(\Phi)}^\rhd)$ where $\alpha = \rho(\Phi)$,
    \item $\rho(P_\Phi^\rhd)=\sum_{\rho_T^\rhd(\Phi)\leq\alpha\leq\rho_C^\rhd(\Phi)}\mu_{\mathcal{P}}(\alpha,\rho_C^\rhd(\Phi))F_\alpha$.
\end{itemize}
\end{theorem}

This follows from the proofs of Theorems \ref{thm:toqsym} and \ref{thm:pwrsmim} by changing the total order.

\begin{proof}[Proof of Theorem \ref{thm:totcob}]
The proof in Section \ref{sec:proof} requires Lemma \ref{lem:djsets}, Theorems \ref{thm:toqsym} and \ref{thm:pwrsmim}, and for the total order $\rhd$ to project under $\rho$ to $\prec$. We define a total order on set compositions, $\rhd$, that projects to $\prec$ by 

\[
A\rhd B = \begin{cases} |A|\prec |B| & \text{ if } |A|\neq |B|\\
\min(A)<\min(B) & \text{ if } |A|=|B|
\end{cases}.
\]

$\rhd$ is a standard invariant projective total order, thus a total order analogue of Lemma \ref{lem:djsets} and Theorem \ref{thm:totpro} holds. 
\end{proof}

In Section \ref{sec:faminv} we could naturally pair total orders (and thereby pairing powersums) using the star involution. Now we will show how to do this generally.

Recall that the \textit{algebraic complement involution}, denoted $\leftarrow$, sends $M_\Phi$ to $M_{\overleftarrow{\Phi}}$, where $\Phi=(B_1,\ldots,B_k)$ and $\overleftarrow{\Phi}=(B_k,\ldots,B_1)$ is the reverse set composition of $\Phi$.

\begin{theorem} \label{thm:alginv}
The algebraic complement of a quasipowersum is the reverse powersum
\begin{equation}
    \overleftarrow{P_\Phi^\rhd}=P_{\overleftarrow{\Phi}}^{\overleftarrow{\rhd}}
\end{equation}
where $\overleftarrow{\rhd}$ is the reverse of the total order $\rhd$.
\end{theorem}

\begin{proof}
By definition, $\overleftarrow{P_{\Phi}^\rhd}=\overleftarrow{\sum_{\Psi\rhd\Phi}M_\Psi}=\sum_{\Psi\rhd\Phi}M_{\overleftarrow{\Psi}}$. Denote the reverse composition as $\overleftarrow{\Phi}=A_1|\cdots|A_k$. Note that for blocks $B_i,B_{i+1}$ in $\Phi$, if $B_i\rhd B_{i+1}$, then $\overleftarrow{B_i}\rhd\overleftarrow{B_{i+1}}$, which means that 
$A_{k+1-i}\lhd A_{k+1-i+1}$, which is rewritten as $A_{k+1-i+1}\overleftarrow{\rhd} A_{k+1-i}$. Thus,  $\sum_{\Psi\rhd\Phi}M_{\overleftarrow{\Psi}}=\sum_{\overleftarrow{\Phi}\overleftarrow{\rhd}\Psi}M_\Psi=P_{\overleftarrow{\Phi}}^{\overleftarrow{\rhd}}$.
\end{proof}

Let $\Phi$ be a set composition of $[n]$. Recall that the \textit{coalgebraic complement involution} is defined as the map $\overline{M_\Phi}=M_{\overline{\Phi}}$, where $\overline{\Phi}=(\overline{B_1},\ldots,\overline{B_k})$ and $\overline{B_i}=\{ n+1-b_{i_1},\ldots,n+1-b_{i_j} \} $. For example, $\overline{M_{13|2|45}}=M_{35|4|12}$. Let $A,B$ be two disjoint subsets of $n$, we define the complement of a shift-invariant total order $\rhd$, denoted $\overline{\rhd}$, as 
\[
A \overline{\rhd} B = \begin{cases} \overline{A}\lhd\overline{B} & \text{ if }|A|=|B|\\ A\rhd B & \text{ if } |A|\neq |B| \end{cases}.
\]

We similarly define the complement of shift-invariant total order $\rhd$ as $A\rhd B =\overline{A}\lhd\overline{B}$.

\begin{proposition}
If $\rhd$ projects to $\prec$ under $\rho$, then $\overline{\rhd}$ does as well. In other words complementation is an involution on such total orders.
\end{proposition}

This follows from the definition that $A\overline{\rhd}B$ when $A\rhd B$. This goes to say that if a total order $\rhd$ compares to disjoint sets based off the sets' size then the complement total order will still compare the sets by using the sets' size.

\begin{theorem}
Let $\Phi$ be a set composition of $[n]$. If $\rhd$ is a total order on set compositions, then
\[
\overline{P_\Phi^\rhd}=P_{\overline{\Phi}}^{\overline{\rhd}}
\]
where $\overline{\rhd}$ is the complement of the total order $\rhd$.
\end{theorem}

\begin{proof}
Let $\Phi=(B_1,\ldots,B_k)$. This proof will be done in two cases, when $\rhd$ is projective, and when $\rhd$ is not projective. However for both cases $\overline{P_\Phi^\rhd}=\overline{\sum_{\Psi\rhd\Phi}M_\Psi}=\sum_{\Psi\rhd\Phi}M_{\overline{\Psi}}$.

Let $\rhd$ be projective. If $B_i\rhd B_{i+1}$, then $\overline{B_i}\lhd\overline{B_{i+1}}$ when $|B_i|=|B_{i+1}|$, which means that $B_i\overline{\rhd}B_{i+1}$. And when $|B_i|\neq|B_{i+1}$, if $B_i\rhd B_{i+1}$, then $B_i\overline{\rhd} B_{i+1}$. Thus it is equivalent to write $\sum_{\Psi\rhd\Phi}M_{\overline{\Psi}}=\sum_{\Psi\overline{\rhd}\overline{\Phi}}M_{\Psi}=P_{\overline{\Phi}}^{\overline{\rhd}}$.

Finally if $\rhd$ is not projective, then If $B_i\rhd B_{i+1}$, then $\overline{B_i}\lhd\overline{B_{i+1}}$ for all $B_i$.
Thus $\sum_{\Psi\rhd\Phi}M_{\overline{\Psi}}=\sum_{\Psi\overline{\rhd}\overline{\Phi}}M_{\Psi}=P_{\overline{\Phi}}^{\overline{\rhd}}$.
\end{proof}

\section{Powersums in Other Algebras} \label{sec:othralg}

In \cite{bergeron2021hopf}, the authors define a way to define a monomial basis for a combinatorial Hopf algebra given certain conditions. In some of the examples below we sketch the relationship that $P$ has with other combinatorial algebras using fillings and the monomial basis.

\subsection{$\fqsym$}
$\fqsym$ is the Hopf algebra of permutations first introduced by Malvenuto and Reutenauer \cite{malvenuto1995duality}, which is a subalgebra of $\ncqsym$ \cite[Section 3]{duchamp2002noncommutative} \cite[Section 2.4]{novelli2010combinatorial}. The image of $\fqsym$ on NCQSym is the powersums basis indexed by singleton set compositions. To show this we recall the Hopf algebra of Word Quasisymmetric functions (WQSym), which is indexed by packed words and is isomorphic to NCQSym.

A packed word of $n$ is a word $w = w_1\cdots w_n$ such that for any $w_i$ either the letter $w_i-1$ is in $w$ or $w_i=1$. Then packing a word $u$, $pack(u)$, is the packed word $w$ such that all letters keep the same relative order as $u$. The space of WQSym is spanned by the monomial basis where $M_w = \sum_{pack(u)=w}u$. The map $\setcomp(w)$ bijects packed words to a set composition by $w= w_1\cdots w_n\to\Phi = B_1|\cdots|B_k$ where $i\in B_{w_i}$. The map $\setcomp:M_{w}\to M_{\setcomp(w)}$ is a Hopf isomorphism between WQSym and NCQSym.

The standardization of a word, $st(w)$ (which is not to be confused with standardization of a set composition $\Phi^\downarrow$), is the permutation obtained from reading the word from left to write and label 1, 2, ... when $w_i=1$ and repeat for $w_i = 2$ and so on. For example $132$ and $121$ are both packed words with $st(132)=st(121) = 132$. Let $\Phi = (B_1,\ldots,B_k)$ and $\Psi = (A_1,\ldots,A_l)$ be a set composition and a singleton set composition respectively, $st(\Phi) =\Psi$ means that for every $i$, $B_i = A_j\cup\cdots\cup A_{j+j'}$ where $A_j<\cdots<A_{j+j'}$ where $<$ compares the only integer in the block.

Explicitly, FQSym is a sub Hopf algebra of WQSym by the map $\mathcal{G}_\tau\to\sum_{st(u)=\tau} M_u$ where 
$\mathcal{G}_\tau$ is the dual fundamental basis and $\tau$ is a permutation. Thus for NCQSym, $\mathcal{G}_\tau\to\sum_{st(\Phi)=\tau(1|\cdots|n)} M_\Phi$. We note that the $\Phi$ in this sum are precisely the set compositions satisfying $\Phi \geq_{\widetilde{\mathcal{D}}} \tau(1|\cdots|n)$, thus by Definition \ref{def:wposet}, 
\begin{theorem} \label{thm:fima}
The image of the $\mathcal{G}$ basis on $\ncqsym$ is the quasisymmetric powersum function in non-commuting variables indexed by a singleton set composition, i.e. for a permutation $\tau$
\begin{equation}\label{eqn:fima}
    \mathcal{G}_\tau = P_{\tau(1|2|\cdots|n)}
\end{equation}
\end{theorem}

\subsection{QSym$^r$}
Hivert defined the space of $r$-quasisymmetric functions, denoted $\qsym^r$ where $r$ is some nonnegative integer, as the invariant space of certain local actions in \cite{hivert2004local}. A local action is a permutation from $\mathfrak{S}_n$, where $n$ is the number of variables, that acts on the polynomial ring by 
\begin{equation} \label{eqn:lclact}
\sigma_i(x_i^a x_{i+1}^b)=
\begin{cases} 
x_i^b x_{i+1}^a & \text{if }a<r \text{ or } b<r\\
x_i^a x_{i+1}^b & \text{else}
\end{cases}
\end{equation}
 
The spaces of $r$-quasisymmetric functions are nested sub Hopf algebras of $\mathbb{Q}[X]$, in other words $\mathbb{Q}[X] =\qsym^0 \supset \qsym=\qsym^1 \supset \qsym^2 \supset \ldots \supset \qsym^r \supset \ldots \supset \sym $.

Let $\alpha=(a_1,\ldots,a_k)$ be a composition of $n$ such that $a_i\geq r$ and let $\lambda=(\lambda_1,\ldots,\lambda_l)$ be a partition of $m$ such that $\lambda_i< r$. Then we define an \textit{$r$-composition} of $n+m$ as the pairing of $\alpha$ and $\lambda$ denoted as $(\alpha;\lambda)$. Let $\beta$ be a composition of $n$, then $\rsrt(\beta)$ is the $r$-composition that comes from moving all parts less than $r$ to the end. For example $\sort_3((3,5,2,1,3,4,5,2))=(3,5,3,4,5;2,2,1)$.

The $r$-quasimonomial function, denoted as $M_{\alpha,\lambda}^r$ is a basis of $\qsym^r$ and is defined as 
\begin{equation*} 
    M_{\alpha,\lambda}^r = \sum_{\beta:\rsrt(\beta)=\alpha;\lambda}M_\beta.
\end{equation*}
In much the same way as above, we can define fillings for QSym$^r$, denoted $\sd^r$, to define a combinatorial way to change the basis from $P$ to $M$.
\begin{equation*} 
    P_{\alpha,\lambda}^r=\sum_{\substack{\mathsf{F}\in \sd^r(\alpha;\lambda) \\\sigma \in \mathfrak{S}_\mathsf{F}}}M_{\col(\sigma(\mathsf{F}))}^r.
\end{equation*}
This also has the property that $P_{\alpha;\mu}^r=\sum_{\rsrt(\beta)=\alpha;\mu}P_\beta$. This will be explored more in a future paper.

\subsection{Cyclic Quasisymmetric Functions}

In \cite{adin2021cyclic} the authors defined a different action on QSym called the cyclic shift, denoted $\varsigma_i$. The invariant space of QSym under the cyclic shift is the space of cyclic quasisymmetric functions, cQSym, note that this isn't a Hopf algebra, due to the fact that cQSym is not preserved under the coproduct. In \cite{adin2021cyclic}, the authors use set notation, however out of familiarity we will use composition notation. Let $\alpha = (a_1,\ldots,a_k)$ be a composition, then the cyclic shift of $\alpha$ is $\varsigma(\alpha) = (a_k,a_1,\ldots,a_{k-1})$. Thus the monomial basis of cQSym is 
\begin{equation}\label{eqn:cQSym}
    \Hat{M}_\alpha = \sum_{i = 1}^k M_{\varsigma^i(\alpha)}.
\end{equation}

For example, $\Hat{M}_{(3,1,2)} = M_{(3,1,2)}+ M_{(2,3,1)} + M_{(1,2,3)}$ and $\Hat{M}_{(2,1,2,1)} =  2M_{(2,1,2,1)} + 2M_{(1,2,1,2)}$. The powersums for cQSym can also be defined in terms of fillings (with an equivalence relation on fillings by cycling the columns) and row permutations
\[
\hat{P}_\alpha = \sum_{\substack{\mathsf{F}\in\hat{\sd}(\alpha)\\ \sigma\in\Hat{\mathfrak{S}}_{\mathsf{F}}}}M_{\col(\sigma(\mathsf{F}))}.
\]
For example, $\Hat{P}_{(2,1,2,1)}=2\Hat{M}_{(2,1,2,1)}+4\Hat{M}_{(3,1,2)}+2\Hat{M}_{(3,3)}$. Contrary to the scaling of \eqref{eqn:cQSym}, the cyclic quasipowersums has either a coefficient of 1 or 0 when expanding a cyclic quasipowersum in terms of quasipowersums. From the example above, $\Hat{P}_{(2,1,2,1)} = P_{(2,1,2,1)}+P_{1,2,1,2}$. Thus the alternate definition is
\[ 
\hat{P}_\alpha = \sum_{i\in I}P_{\varsigma^i(\alpha)}.
\]
where $I = [j]$ where $j$ is the minimum integer such that $\varsigma^{j+1}(\alpha) = \varsigma(\alpha)$.

\printbibliography

\end{document}